\documentclass[12pt]{amsart}
\usepackage{amssymb}
\usepackage{graphicx}
\usepackage{xcolor} % A package to add color.
\usepackage{tensor}

\usepackage{fullpage} % Sets all margins to 1 in.  

\definecolor{green}{rgb}{0,0.8,0} % Redefines the color green.
\newtheorem{theorem}{Theorem}[section]

\newtheorem{lemma}[theorem]{Lemma}
\newtheorem{proposition}[theorem]{Proposition}
\theoremstyle{definition}

\theoremstyle{remark}
\newtheorem{remark}[theorem]{Remark}
\numberwithin{equation}{section}
\newcommand{\abs}[1]{\vert#1\vert^2}

\newcommand{\pa}{\partial}
\newcommand{\na}{\nabla}

\newcommand{\del}{\delta}
\newcommand{\Del}{\Delta}
\newcommand{\ep}{\epsilon}

\newcommand{\la}{\lambda}

\newcommand{\om}{\omega}
\newcommand{\Om}{\Omega}
\newcommand{\ps}{\psi}

\newcommand{\bbR}{\mathbb R}
\newcommand{\bbr}{\mathbb R}

\newcommand{\calP}{\mathcal P}

\newcommand{\calU}{\mathcal U}

\newcommand{\ddt}{\frac{d}{dt}}
\newcommand{\vp}{\varphi}
\newcommand{\dvg}{\na \cdot}
\newcommand{\dv}{\na \cdot}
\newcommand{\fn}[1]{\frac{ #1}{N}}

\newcommand{\lr}[1]{\left( #1 \right)}
\newcommand{\ipa}{\pa_y^{i}}
\newcommand{\jpa}{\pa_z^{j}}
\newcommand{\Fn}{\left(\fn{1} \right)}
\newcommand{\llangle}{\langle\!\langle}	% double left angled bracket
\newcommand{\rrangle}{\rangle\!\rangle}	% double right angled bracket
\newcommand{\cp}{\calP}
\newcommand{\cpp}{\calP'}
\newcommand{\cpq}{\calP''}
\newcommand{\np}{N'}

\newcommand{\infz}[1]{\int_{-\infty}^z{#1} dz}
\begin{document}
\title[Stability of traveling waves in Keller-Segel]{Stability of planar traveling waves \\ in a Keller-Segel equation \\
on an infinite   strip domain}%: Title of the article
%\title[]{Transversal stability of Traveling waves \\ in a Keller-Segel equation \\
%{on the surface of a thin cylinder}}%: Title of the article
%\author{Myeongju Chae}%
%
%\author{Kudong Choi}%
%\address{Department of Mathematics, UW Madison, WI}%
%\email{kchoi@math.wisc.edu}%

%\thanks{}%

%\date{\today}%
%\dedicatory{}%
%\commby{}%
% ---------
\author{Myeongju Chae}%
\address{Department of Mathematics, Hankyung University, Anseong-si, Gyeonggi-do, Republic of Korea}%
\email{mchae@hknu.ac.kr}
\author{Kyudong Choi}%
\address{Department of Mathematical Sciences, Ulsan National Institute of Science and Technology, Ulsan, Republic of Korea}%
\email{kchoi@unist.ac.kr}
\author{Kyungkeun Kang}%
\address{Department of mathematics, Yonsei University, Seoul, Republic of Korea}%
\email{kkang@yonsei.ac.kr}
\author{Jihoon Lee}%
\address{Department of Mathematics,
Chung-Ang University,
Seoul, Republic of Korea
}%
\email{jhleepde@cau.ac.kr}
\begin{abstract}
A simplified model of the tumor angiogenesis can be described by  a Keller-Segel equation \cite{FrTe,Le,Pe}. The   stability of traveling waves for the one dimensional system has recently been known by \cite{JinLiWa,LiWa}.
In this paper
we consider the equation on the two dimensional 
 domain 
$ (x, y) \in  
 \bbr \times {\mathbf S^{\la}}$ 
for 
a  small parameter $\la>0$
where $ \mathbf S^{\la}$ is the circle of perimeter $\la$. Then the equation allows a planar traveling wave solution of invading types.
We establish the nonlinear stability of the traveling wave solution if the initial perturbation 
is  sufficiently small  in a weighted Sobolev space  without a chemical diffusion. 
When the diffusion is present, we show a linear stability. Lastly, we prove that any solution with our front conditions  eventually becomes planar under certain regularity conditions.
The key ideas are to use the Cole-Hopf transformation  and  to apply 
the Poincar\'{e} inequality   
 to handle with the two dimensional structure.

\end{abstract}
\maketitle
% ----------------------------------------------------------------
\section{Introduction and main theorems}
{The formation of new blood vessels (angiogenesis) is the essential mechanism for tumour progression and metastasis.}
A simplified model of the tumor angiogenesis  can be described by  a Keller-Segel equation \cite{FrTe,Le,Pe}.
In this paper
we consider the equation on the two dimensional cylindrical domain $ (x, y) \in  \Om = \bbr \times [0, \la]$
with a front boundary condition in $x$  and the periodic condition in $y$, both specified later,
\begin{align}\label{KS} \begin{aligned}
\pa_t n - \Del n& = - \na \cdot (n \chi(c) \na c ), \\
\pa_t c - \ep\Del c & = - c^m n
\end{aligned} \end{align}
with $m \ge 1$, where $n(x, y, t) \ge 0 $ denote the density of endothelial cells, $c(x,y, t) \ge 0$ 
stands for 
the concentration of the chemical substance or the protein known as the vascular endothelial  growth factor (VEGF) and $\chi(\cdot): \bbr^+ \to  \bbr^+ $ is a decreasing chemosensitivity  function, 
reflecting that the chemosensitivity is lower for higher concentration of the chemical. 
The system \eqref{KS} includes the both zero chemical diffusion $(\ep=0)$ and non-zero chemical diffusion $(\ep>0)$ cases.\\
\indent 
Endothelial cells forming the linings of the blood vessels are responsible for extending and remodeling the network of 
blood vessels, tissue growth and repair. Tumors or cancerous cells are also dependent on bloods supply by newly  
generated capillaries formed toward them, which process is called the endothelial angiogenesis. 
In the modeling the  endothelial angiogenesis, the biological implication is that the  endothelial cells
behaves as a invasive species, responding to signals produced by the tissue. 
\footnote{{In  Bruce Alberts et al \cite{Ce}, it is summarized as follows: ``Cells that are short of oxygen increase their concentration of certain protein (HIF-1), which stimulates the production of vascular endothelial growth factor (VEGF). VEGF acts on endothelial cells, causing them to proliferate and invade the hypoxic tissue to supply it with new blood vessels."}}
%Hence the propagation has a profile of  a wavefront \comment{reference}. 
Accordingly the system \eqref{KS} is given the front condition at left-right ends such that
\begin{align}\label{nfront}
\lim_{x\to -\infty} n (x, y, t) = n_- >0 ,  \quad  \lim_{x\to \infty} n(x, y, t) = 0,\\
\lim_{x\to -\infty} c (x, y, t) = 0,  \quad  \lim_{x\to \infty} c(x, y, t) = c_+ > 0. \label{cfront}
\end{align}
We choose the $x$-axis by the propagating direction.

\indent
A \textit{planar}  traveling wave solution of \eqref{KS} is a traveling wave solution 
independent of the transversal direction $y$ such that
\begin{align}\label{twave}
 n(x,y,t) =N(x-st), \quad c(x, y, t) = C(x-st)
\end{align}
with  a wave speed $s >0$. From now on, we consider only planar
 traveling waves  
%always assume that 
 $(N, C)$ satisfying the above boundary conditions \eqref{nfront} and \eqref{cfront}, and 
moreover we assume 
\begin{equation}\label{derivative}
N'(\pm \infty) =  C' (\pm \infty) =0.
\end{equation}
(We denote $\displaystyle \lim_{x\to \infty} N(x)$ by  $N(\infty)$ in short.) 
%The profile of $N$ is of a wavefront type which describe the endothelial cell being invasive species.
% In two spatial dimension, 
%this kind of  traveling wave  is called a \textit{planar} type or a \textit{line soliton} in other literatures.
\\ \indent
In this paper we study the  stability of a planar traveling wave solution $(N, C)$ of \eqref{KS}
 for $\ep\geq0$ case  assuming $\chi(c)= c^{-1}$, $m =1$,
\begin{align}\label{eq:main}
 \begin{aligned}
\pa_t n - \Del n& = - \na \cdot \left(n \frac{ \na c}{c} \right), \\
\pa_t c-\ep\Del c  & = - cn, \quad (x, y) \in \bbr\times [0, \la].
\end{aligned} \end{align}
\indent  { We show that the planar traveling wave solution is asymptotically stable for certain small perturbations in a weighted Sobolev space  when $\ep=0$ (Theorem \ref{theoremnc}), and linearly stable when $\ep>0$(Theorem \ref{theoremli}).
 The restriction on $m =1$ is required for  treating the singularity of $ 1/c$ by the Cole-Hopf transformation, as is presented in Section \ref{hopf}. On the other hand, the nonexistence of traveling wave solutions was proved in \cite{Sc} when $m >1$.
 \\
\indent
%The presence  of traveling wave solution can be a \comment{Fill in} in mathematical model describing a biological 
%phenomena. [ Murray ...]. 
The traveling wave solution with an invading front 
%in \eqref{KS} 
can be an evidence of the tumor encapsulation \cite{AC, BT, Sh}.  
The existence of traveling wave solutions for a Keller-Segel model has been initiated by Keller and Segel \cite{KSb} and was followed by many works. For instance, see \cite{Ho1} and the references therein. It is known  that the chemosensitivity function $\chi(\cdot_c)$  needs to be singular as $c$ approaches to zero for a traveling wave solution to exist (\textit{ e.g.} see\cite{KSb, Sc}). The paper \cite{KSb} and many others assume $\chi(c) = c^{-1}$ when studying traveling waves, which choice of  $\chi(c)$  is also adopted when modeling the formation of the vascular network toward cancerous cells  \cite{FrTe,Le,Pe}.}
 {More general $\chi(c)$ other than $1/c$ together with the various production rates are studied in \cite{HoSt}.}

\begin{subsection}{Main Theorems}
  Let us introduce the weight function:
\[ w(x) = 1 +e^{sx}, \quad x\in \bbr\]
and define the Sobolev spaces $H^k$ and $H^k_w$ by their norms: 
\begin{align*}
&\| \varphi\|_{H^k}^2  = \sum_ { i+ j \le k}\int_{\bbr\times [0, \la]} | \pa_x^i \pa_y^j \varphi (x, y) |^2 dxdy, \quad\\
& \| \varphi\|_{H^k_w}^2  = \sum_ { i+ j \le k}\int_{\bbr\times [0, \la]} | \pa_x^i \pa_y^j \varphi (x, y) |^2 w(x) dxdy, \quad k = 0, 1, 2 \dots.
 \end{align*}
  \indent
% We will assume the Dirichlet boundary condition (in $y$-direction)  for initial perturbation when we consider non-zero chemical diffusion ($\ep>0$) case 
% and use the function spaces $H^k_{ 0}$ and $H^k_{w, 0}$:
% \[  H^k_{0} : =  \overline{ C_0^{\infty}(\Omega)}^{H^{k}},\quad  H^k_{w,0} : =  \overline{ C_0^{\infty}(\Omega)}^{H^{k}_w}, \quad \Omega= \bbr\times (0, \la), 
% \quad k
% \in \bbZ  \mbox{ with } k\ge 0\] where $C^\infty_0(\Om)$ is the space of all  smooth functions which are compactly supported in $\Om$.\\
% 
% On the other hands, 
We will assume the 
  periodic boundary condition in $y$-direction for initial perturbation 
%  when we consider zero chemical diffusion case ($\ep=0$) 
  and use 
   the Sobolev   spaces  $H^k_{p}$ and  $H^k_{w, p}$ :
\begin{align*}   H^k_{p}  & :=  {H^k} ( \bbr\times \mathbf S^{\la}) = \{ \vp\in {H^k} (\Omega)\, |\, 
 \quad \sum_{i+j\le k} 
\sum_{n\in \mathbb{Z}}
 \int_\bbR  n^{2j}|\pa_x^i \varphi_n(x)|^2  dx < \infty \},
\end{align*}
\begin{align*}   H^k_{w,p}  & :=  {H^k_w} ( \bbr\times \mathbf S^{\la}) = \{ \vp\in {H^k_w} (\Omega)\, |\, 
 \quad \sum_{i+j\le k} 
\sum_{n\in \mathbb{Z}}
%\sum_{n\in \mathbb{Z}/{\la}}
 \int_\bbR  n^{2j}|\pa_x^i \varphi_n(x)|^2 w(x) dx < \infty \}
\end{align*}
 where $ \mathbf S^{\la}$ is the circle of perimeter $\la$ %$\la/ 2\pi$
 and $ \varphi_n(x)$ is the $n$th Fourier coefficient of $\varphi(x, \cdot_y )$. \\

In other words, $\vp\in {H^k_{w,p}} $ is a function on $\mathbb{R}^2$ such that it is $\la$-periodic in $y$ and its (weighted in $z$-direction) $H^k$ norm on $\Om=\mathbb{R}\times[0,\la]$ is finite. 
The cylindrical domain with this periodic boundary condition fits well to the picture
that blood vessels are lined with a single layer of endothelial cells. What it follows we use the equivalence of $\| \cdot \|_{H^k_p}$ and  $\| \cdot \|_{H^k_{w,p}}$ norms
to  $\| \cdot \|_{H^k}$ and  $\| \cdot \|_{H^k_{w}}$ norms for functions periodic in $y$.\\
  \indent
We perturb the system \eqref{eq:main} around a planar traveling wave solution $(N, C)$ with the front conditions 
\eqref{nfront}, \eqref{cfront} and  \eqref{derivative}  in a way that 
\begin{align}\label{ncp}
n(x, y, t) &= N(x- st) + u (x-st, y, t) \quad \mbox{and } 
\quad 
&c(x, y, t) = C(x- st)e^{- \psi(x-st, y, t)}  
 \end{align} 
where the traveling wave speed 
$s= \sqrt{
{\frac{n_-}{1+\ep}}}>0
$.
What it follows,
we will use the moving frame \[z=x-st.\] 
%In $z$ variable, it is equivalent to
%\begin{align*}\label{ncp_}
%n(z+st, y, t)&= N(z) + u (z, y, t) \quad \mbox{and } 
%\quad c(z+st,y,t)= C(z)e^{- \psi(z, y, t)}. \nonumber
%\end{align*}
We consider the above perturbation $(u, \psi)$ in  the following function class.
\begin{enumerate}
\item  There exists a vector-valued function $\varphi\in (H^3_{w,p})^2$ % for $\ep=0$ (or $\varphi\in (H^3_{w,0})^2$ for $\ep>0$)
 such that $u = \dv \varphi=\pa_z\vp^1+\pa_y\vp^2$.
\item  $\psi$ lies on $H^3_{p}$ %for $\ep=0$ (or $\ps\in H^3_{0}$ for $\ep>0$) 
with $\na \psi \in H^2_{w,p}$.
\end{enumerate}
 \indent
% Throughout the paper,
%   we always assume that our perturbation $(\vp,\ps)$ in the form of
%$ N - n = \dv \varphi$ and $c /C = e^{-\psi} $ satisfies the periodic condition 
%in $y$-direction with period $\la$.
In the subsection \ref{hopf}, the nonlinear perturbation $(\vp,\ps)$ is found directly related to the  Cole-Hopf transformation.\\ %: periodic for $\ep=0$ (or Dirichlet for $\ep>0$). Together with Sobolev regularity, we consider 
% $(\vp,\ps)\in (H^3_{w,p})^2\times  H^3_{p}$ for $ \ep=0$  
% (or  $(\vp,\ps)\in (H^3_{w,0})^2\times  H^3_{0}$ for $ \ep>0$)    with $\na \psi \in H^2_w$. 
 % We denote  both spaces $H^k_{w, p}$ for $ \ep=0$  and  $H^{k}_{w, 0}$ for $\ep >0$  by $H^k_w$ for notational convenience. 
 % In fact, for  \textit { a priori } estimate  concerned, only the norm of $\| \cdot \|_{H^k_w}$ appears.}

We are ready to state the main results of the paper. The first result is  on the nonlinear stability of the planar traveling
wave solution $(N, C)$ of the system \eqref{eq:main} when $\ep=0$.
 \begin{theorem}\label{theoremnc}
For a given planar traveling wave solution $(N,C)$ of the  system \eqref{eq:main} for $\ep=0$
with \eqref{nfront}, \eqref{cfront} and  \eqref{derivative},
there exist constants % $\ep_0>0,\, 
$m_0>0,\, C_0>0,$ and $\la_0>0$ such that
%if $ 0\le \ep < \ep_0$ and
 if $0<\la\le \la_0$, then
for any
 initial data $(n_0, c_0)$ in the form of $ N - n_0 = \dv \varphi_0$ and $c_0 /C = e^{-\psi_0} 
 $ where $\vp_0$ and $\ps_0$ are $\la-$periodic in $y$ with % and where they satisfy %satisfying 
% $\varphi_0 \in H^3_w $ and $\psi_0 \in H^3$ with $\nabla\psi_0\in H^2_w$. Then  such that if
  $M_0:=\| \vp_0\|^2_{H^3_w} + \| \ps_0\|^2_{H^3}  + \|\nabla\psi_0\|^2_{H^2_w}  \color{black}\le m_0$,  
% the length of transversal direction $\la$ is sufficiently small,
 there exist a unique global classical solution $(n,c)$  of  \eqref{eq:main} for $\ep=0$ in the form of
\[
n(x, y, t) = N(x- st) + \dv\varphi (x-st, y, t), 
\quad 
c(x, y, t) = C(x- st)e^{- \psi(x-st, y, t)},\]
%The perturbation $\varphi$ and $\psi$ satisfy the system 
%\begin{align}\label{eq0:phipsi} \begin{aligned}
%\varphi_t - s \varphi_x - \Del \varphi
%& =   N \na \psi + P\na \cdot \varphi + \na\cdot \varphi \na \psi \\
%\ps_t - s \ps_x -\ep \Del \ps &=  -2\ep  P \cdot\na \ps -\ep  |\na \ps|^2 + \dv\varphi.
%\end{aligned}
%\end{align}
%with $ \varphi(\cdot, 0) = \varphi_0$ and $\psi (\cdot, 0) = \psi_0$.
where $\vp$ and $\ps$ are $\la-$periodic in $y$ for each $t>0$ and $ \varphi|_{t=0} = \varphi_0$ and $\psi|_{t=0} = \psi_0$.
Moreover, $ (\phi, \psi)$ satisfies the following inequality:
%  \marginpar{add \\$\| \na \psi \|_{H^2_w}$}
\begin{align}\label{energyineq} \sup_{t\in [0, \infty)}\Big(\| \varphi\|^2_{H^3_w} + \| \psi\|^2_{H^3} + \| \na \psi \|^2_{H^2_w}\Big) 
%+ \int_0^{\infty} \int  \sum_{1\le k \le 4} \frac{|\na^k \varphi|^2}{N} + \int_0^{\infty} \int \sum_{1\le k\le 3} N|\na^k \psi|^2 + \frac{|\na^k \psi|^2}{N}
+\int_0^{\infty}  \Big(  \| \na  \vp\|_{H^3_w}^2
+   \| \na  \ps\|_{H^2_w}^2
%+\ep     \| \na^{4} \ps\|_w^2 
\Big)dt
 \le C_0M_0.
 \end{align}
\end{theorem}
\begin{remark}
 \begin{enumerate}
% \item The existence of planar traveling waves with \eqref{nfront}, \eqref{cfront} and  \eqref{derivative} will be described in the next subsection
% \ref{existence_traveling}.
 \item
 The condition $ u= \dv \vp $  implies  the zero integral condition on $u$, $\int u dzdy = 0$.
Equivalently,  $n$ and $N$  satisfy 
\[ \int n(x, y, t) - N(x-st)  dxdy = 0,\]
which is preserved over time, provided  it holds at $t=0$ by the mass conservation property of $n$ equation.
\item
The weight function $w$ is of a same order as $1/N$ in $z>0$ direction. More precisely, for a fixed traveling wave, there exists a constant $M$ satisfying %there exist 
%$\ep_0>0$ and  $C$ independent of $\ep\in[0,\ep_0]$ such that
\[ \frac{w}{M}\leq\frac{1}{N}\leq Mw,\]
which is proved in Section $2.4$.

\item   {
%\fbox{Dec. 16} 
%In the zero diffusion case 
Theorem \ref{theoremnc} implies the asymptotic convergence of 
the primary perturbation of $(n, c)$ to $(N, C)$. From
 \[   
 \int_0^{\infty}  \Big(  \| \na  \vp\|_{H^3_w}^2
+   \| \na  \ps\|_{H^2_w}^2
%+\ep     \| \na^{4} \ps\|_w^2 
\Big)dt
 \le C_0M_0,
 \] 
 we get a weak asymptotical estimate
  \[   
\liminf_{t \rightarrow \infty}  \Big(  \| \na  \vp(t)\|_{H^3_w}^2
+   \| \na  \ps(t)\|_{H^2_w}^2
%+\ep     \| \na^{4} \ps\|_w^2 
\Big) 
= 0.\] In terms of $n(\cdot_z +st )-N=\dv\vp $ and $c(\cdot_z +st)/C=e^{-\psi}$, %$c-C=C(e^{-\psi}-1)$,
 we have
% \[   
%\liminf_{t \rightarrow \infty}  \Big(  \|n(\green{\cdot_z+st,\cdot_y,t)-N(\cdot_z)  \|_{H^3_w}^2
%+   \|  \nabla\Big( \frac{ c(\cdot_z+st,\cdot_y,t)}{C(\cdot_z) }\Big)   \|_{H^2_w}^2+
%\|  \nabla\Big( \frac{C(\cdot_z) }{c(\cdot_z+st,\cdot_y,t)  }\Big) }  \|_{H^2_w}^2
%%+\ep     \| \na^{4} \ps\|_w^2 
%\Big) 
%= 0.\]
 \[   
\liminf_{t \rightarrow \infty}  \Big(  \|n(\cdot_z +st,\cdot_y,t)-N(\cdot_z)  \|_{H^3_w}^2
%+   \|  \nabla\Big( \frac{c (\cdot_z +st,\cdot_y,t)}{C(\cdot_z) }\Big)   \|_{H^2_w}^2 
+   \|  \nabla\big( 
\log c (\cdot_z +st,\cdot_y,t) - \log C(\cdot_z)
\big)   \|_{H^2_w}^2
%+\|  \nabla\Big( \frac{C(\cdot_z) }{c(\cdot_z+st,\cdot_y,t)  }\Big)  \|_{H^2_w}^2 
\Big)
= 0.\]}
\item
The one-sided  decay appears naturally  with respect to the solvability of $\dv \vp = u$ in the infinite strip domain $\Omega$ \cite{Sol}.
\footnote{In the infinite strip domain $\Omega$, it is shown that for any $u$ satisfying $ \| u \|_{L^2(\Omega)}^2 
  + \int_0^{\infty} \mathcal U^2 (x) dx <\infty$ there exist a vector function $\varphi$ in $H_0^1(\Omega)$ such that $\dv \vp = u$ in \cite{Sol}.}
%we can see $u$ decay fast in $z>0$ direction if 
Indeed, if  there is a vector function $\vp \in L^2(\Om)$ satisfying the divergence equation,
assuming Dirichlet or the periodic boundary condition  on the lateral boundary,
%\[ \vp = 0  \mbox{ on }  (z, 0)\cup (z, \la) \mbox{ for } z \in \bbr.\] 
we have 
\[ \int_{z> x} \int  u(z, y)  dz dy=  \int_{z> x} \int_{[0, \la]}   \dv \vp  dz dy =  -\int_{[0, \la]}  \vp^1( x, y) dy.\]
Define $ \calU (x) = \int_{z> x}\int u dy dz  $, then $\calU(x)$ is found in $L^2([0, \infty))$ because 
\[ \int_0^{\infty} \calU^2(x) dx = \int_0^{\infty} \lr{ \int _{[0, \la]} \vp^1 (x, y) dy}^2 dx \le
\int _0 ^{\infty} \la \int _{[0, \la]} |\vp^1|^2(x, y) dy dx = \la\| \vp^1\|_{L^2}^2.\]
%Hence for $u $ to satisfy  $\dv \varphi = u$, 
%it must have a  spatial decay such that  $\int_0^{\infty} \calU^2  dx < \infty$ holds.
 In this paper, we do not pursuit solving $\dv \vp = u$ for a given perturbation $u$. Instead 
 we are given $\vp \in H^3_w$.
   Note that the condition $ \| \vp\|_{H^3_w} < \infty$ implies $ wu \in  L^{\infty}$,
    which is enough for $u$ to satisfy
  $\int_0^{\infty} \calU^2  dx < \infty$ and to be consistent with the necessary condtion for solvability.
  We refer to  \cite{Sol} for the solvability of the divergence equation on the more general noncompact domain.
  
%  \item
%  \comment{For the moment we only establish the zero diffusion case. Though the diffusion constant of the chemical may be negligible compared to growing speed of capillary networks, 
%  the technical reason not similarly treating $\ep>0$ case is in that
%  a cancelling property \eqref{cancelation} does not hold in this case, which is crucial to using the Poincare inequality and hence controlling
%  $L^2$ energy uniformly in time.}
  \end{enumerate}
\end{remark}

 When the chemical diffusion is present $(\ep>0)$, the above result of Theorem \ref{theoremnc} is hard to obtain for technical reasons. Instead, we obtain  a linear stability of the planar traveling wave solution $(N, C)$ of 
the system \eqref{eq:main} when $\ep>0$ for perturbations in a smaller class(mean-zero in $y$). Indeed, as will be derived in Section \ref{hopf}, the pertubation  $(\varphi, \psi)$ satisfies
\begin{align}\label{eq0:nonlinear}\begin{aligned}
\varphi_t - s \varphi_z - \Del \varphi
& =   N \na \psi + P\na \cdot \varphi + \na\cdot \varphi \na \psi, \\
\ps_t - s \ps_z -\ep \Del \ps &=  -2\ep  P \cdot\na \ps -\ep  |\na \ps|^2 + \dv\varphi
\end{aligned}\end{align} {where  $P :=-(C'/C, 0)$}.
Quadratic terms dropped out, the linearized system of $(\vp, \ps)$ is as follows:
\begin{align}\label{eq0:linear} \begin{aligned}
\varphi_t - s \varphi_z - \Del \varphi
& =   N \na \psi + P\na \cdot \varphi,  \\
\ps_t - s \ps_z -\ep \Del \ps &=  -2\ep  P \cdot\na \ps  + \dv\varphi.
\end{aligned}
\end{align}
What it follows we denote $\bar \phi (z) = \int_0^{\la} \phi(z, y) dy$.
\begin{theorem}\label{theoremli} %\textcolor{blue}{Nov.24}
%Any given planar traveling wave solution $(N,C)$ of the  system \eqref{eq:main} for sufficiently small $\ep>0$
%with \eqref{nfront}, \eqref{cfront} and  \eqref{derivative} is linearly stable in the following sense: \\
\indent
For  sufficiently small $\ep>0$ and for any given planar traveling wave solution $(N,C)$ of the  system \eqref{eq:main}  with \eqref{nfront}, \eqref{cfront} and  \eqref{derivative},
the wave $(N,C)$ is linearly stable in the following sense: \\
 % For $n_->0$ and for $c_+>0$}
 There exist constants 
 $ %\ep_0>0, %m_0>0,
% \, 
 C_0>0,$ and $\la_0>0$ such that
% if $ 0< \ep \leq \ep_0$ and 
 if  $0<\la\le \la_0$, then for any initial data $(\vp_0, \ps_0)$ which is $\la-$periodic in $y$ 
with 
  $$
  \bar\vp_0= \bar \ps_0 =0  \mbox{ and } M_0:=\| \vp_0\|^2_{H^3_w} + \| \ps_0\|^2_{H^3}  + \|\nabla\psi_0\|^2_{H^2_w}  \color{black}<\infty,  $$ there exists a unique global classical solution $(\vp, \ps)$ of \eqref{eq0:linear} which is  $\la-$periodic in $y$   and $\bar\vp= \bar \ps =0$ for each $t>0$ and which satisfies
%  satisfying
   the inequality:
  \begin{align}
  \label{energyineqwith}   \sup_{t\in [0, \infty)}\Big(\| \varphi\|^2_{H^3_w} + \| \psi\|^2_{H^3} + \| \na \psi \|^2_{H^2_w}\Big) 
%+ \int_0^{\infty} \int  \sum_{1\le k \le 4} \frac{|\na^k \varphi|^2}{N} + \int_0^{\infty} \int \sum_{1\le k\le 3} N|\na^k \psi|^2 + \frac{|\na^k \psi|^2}{N}
+\int_0^{\infty}  \Big(  \| \na  \vp\|_{H^3_w}^2
+   \| \na  \ps\|_{H^2_w}^2
%+\ep     \| \na^{4} \ps\|_w^2 
{+\ep     \| \na^{4} \ps\|_{H^0_w}^2} \Big)dt
 \le C_0M_0.\end{align}
\end{theorem}

% {Theorem \ref{theoremnc} turns out  the immediate consequence of Proposition \ref{zeroglobal} stating the existence of the  small global solution of \eqref{eq0:nonlinear} when $\ep=0$. For Theorem \ref{theoremli} considering the linear system \eqref{eq0:linear},
%the only issue is to prove the asymptotic behavior \eqref{energyineqwith}.  \\

The next theorem gives a hint why studying 
planar traveling waves is natural instead of studying general traveling waves in $\Om$.
% explains why we consider not a general traveling wave in $\Om$ but a planar traveling wave in this paper. 
Roughly speaking, any smooth bounded solution  with our front boundary conditions \eqref{nfront}, \eqref{cfront} becomes eventually planar under some strong regularity conditions when $\ep>0$.
  This kind of nonlinear stability can be found in \cite{Const} on the stability of a reactive Boussinesq system in
an infinite vertical strip.  It also hints the asymptotic stability as in Theorem \ref{theoremnc} 
might hold as well 
%on those domains 
in a nonzero diffusion case. We state this theorem in terms of $(n, q)$ in \eqref{nq} which will be obtained  in the subsection \ref{hopf} by applying the Cole-Hopf transformation $q:=-\frac{\nabla c}{c}$ to % $(n, c)$
\eqref{eq:main}. 
\begin{theorem}\label{thm_eventual}
Let  $\ep>0$ and $(n, q)$ be a global smooth solution of \eqref{nq}  {  %for some $\epsilon>0$ 
which is periodic in $y$ 
%. Moreover there exists a constant $C$ such that 
 satisfying 
$$\sup_{t\in[0,\infty)} \Big(\| n(t)\|_{L^{\infty}}
% + \| \na n (t)  \|_{L^{\infty}}
  + \|q (t) \|_{L^{\infty}} + \| \na q (t) \|_{L^{\infty}} \Big)< C_1 $$
for some constant $C_1$. In addition, suppose that the derivatives of $(n,q)$ vanish sufficiently rapidly as $|x|\rightarrow \infty$.
Then there exists a constant $c=c(\ep,\la)>0$ such that %it holds that for all $t >0$ 
%\[ \int_0^\infty \| n_y\|_{L^2} +\ep \| q_y \|_{L^2} dt  < C_2\la ( \| n_y(0)\|_{L^2} + \|q_y(0)\|_{L^2})\]
\[ \| n_y\|^2_{L^2} + \| q_y \|^2_{L^2} \leq  
  ( \| n_y(0)\|^2_{L^2} + \|q_y(0)\|^2_{L^2})e^{-ct}\quad\mbox{ for } t>0\]
if the transversal length $\la>0$ is sufficiently small.}
\end{theorem}
\begin{remark} { It implies
%we get a weak asymptotical estimate
 $   
\lim_{t \rightarrow \infty}  \Big( \| n_y\|_{L^2} + \| q_y \|_{L^2}
\Big) 
= 0.$}
\end{remark}

\indent 
The remaining parts of the paper are organized as follows. In Section  \ref{backg}, we introduce the backgroud materials including the existence of traveling wave solutions and the Cole-Hopf trasformation.
 In Section \ref{sectionuniform}, we prove Proposition
\ref{zerolocal} and \ref{zeroglobal}, then establish Theorem \ref{theoremnc}. 
 In Section \ref{sectionuniformwith}, we prove the energy inequality \eqref{energyineqwith} in 
 Theorem \ref{theoremli}. With \eqref{energyineqwith} the existence part of Theorem \ref{theoremli}  follows   the same argument that works for the $\ep=0$ case, therefore we omitted. In the last section, %\ref{stabilitythin},
  we proves Theorem \ref{thm_eventual} which points
%present an \textit{a  priori} estimate pointing 
out a nonlinear stability on a thin cylindrical domain for   $\ep>0$ case would be expected as well.\\ 
\indent
 There are many prior results on  the existence of traveling wave solution and its stability. 
  Among earlier analytic works is \cite{NaI}, 
  where they proved the existence and linear instability of traveling wave solution of the one dimensional system when the bacterial consumption rate is constant and $\chi(c) = 1/c$.
 The existence of the traveling wave solution with the  front conditions \eqref{nfront} and \eqref{cfront} 
 was shown in \cite{WaHi} when $\ep=0$,  and \cite{LiWa, Wa} when $\ep>0$. 
 % In zero chemical diffusion cases, we refer to \cite{JinLiWa, WaHi} and in non-zero chemical diffusion cases, to \cite{LiWa, Wa, LiLiWa}. 
 In the one dimensional system, the nonlinear asymptotic stability results were shown 
 under the aforementioned restrictions on $\chi(c)$ and $m$
  in \cite{JinLiWa} when $\ep=0$, and \cite{LiWa} when $\ep>0$ is small.
   To our knowledge, our Theorem \ref{theoremnc} 
%   and Theorem \ref{theoremli} 
   is the first result on a nonlinear  stability of   traveling wave solutions in a higher dimension.\\
   \indent
  For the results on the Cauchy problem of \eqref{KS}, see \cite{CPZ1, CPZ2, FoFr, FrTe, LiLi}, where \cite{CPZ1, CPZ2} prove the existence of a global weak solution, and \cite{LiLi} proves that of a global classical solution. Both results consider the zero chemical diffusion case in a multi-dimension. \\
 \indent 
 %We conclude this section with a few remarks.\\
 
 \end{subsection} 
 
 \section{Background}\label{backg}
 
\begin{subsection}{Existence of traveling wave solutions}\label{existence_traveling}
%\comment{In this subsection and next our exposition includes both $\ep =0$ and $\ep>0$ cases for the 
%future reference.}
The traveling wave solution $(N, C)$ in  \eqref{twave} solves the following ODE system,
\begin{align} \begin{aligned} \label{NC}
-sN' -  N'' = -  ( \chi(C)  C' N)',\\
-s C' -  \ep C'' = - C^m N.
\end{aligned} \end{align}
It is easy to see the above system  integrable. We briefly discuss its solvability and properties of the solution  for reader's convenience. What it follows in this subsection has been already well known (for instance in \cite{LiLiWa, Wa}).

Let $H'(\cdot) = \chi(\cdot)$. We first solve $N$ by
%For $S = (s, 0)$ and $H'(\cdot) = \chi(\cdot)$,  the $N$ equation is written as 
%\[ \na \cdot ( S N+ \na N) = \na \cdot (N \na H(C)).\]
%We have 
%\begin{align*}
%&\na \ln N = - \na ( sz - H(C)), \, \ln N = - (s z -H(C))+ N_0,\\
%& N(z, y) = e^{-sz} e^{H(C)} e^{N_0}.
%\end{align*}
\[ N(z) = N_0 e^{-sz} e^{H(C)}\]
for some positive constant $N_0$, {which means a translation by  $z_0$ when $N_0=e^{s z_0}$.}
The system \eqref{NC} with the front conditions \eqref{nfront} and \eqref{cfront} is translation invariant.
It causes the undetermined constant $N_0$.

To satisfy the front condition \eqref{nfront},  we impose $e^{H(C)} \to 0$ as $z \to -\infty$, 
which implies
$H(\cdot)$ is singular at zero such as $\lim_{C \to 0} H(C) = -\infty$.  Substituting  $N$ to the C equation, we have 
\[ sC' + \ep C'' = C^m e^{-sz} e^{H(C)} {N_0}.\]
What it follows we let 
\[ H(C) = \ln C, \quad m=1.\]
 We result in 
\begin{align}\label{NC_copy}
N= e^{-sz} C {N_0}, \quad  sC' + \ep C'' = e^{-sz}C^2 {N_0}.
\end{align}  By introducing 
\[W = e^{-sz}C\]
we have a KPP-type equation for $W$:
\begin{equation}\label{W}
 \ep W'' + (s + 2s\ep) W' = - (\ep +1)s^2 W +{N_0}W^2:  = -f(W).
\end{equation}
Since $ N= WN_0$, front conditions  \eqref{nfront}, \eqref{cfront} force $W$ to have 
\[   \lim_{z \to -\infty} W (z) = \frac{1}{N_0}n_- \quad   \lim_{z \to \infty} W (z) = 0.\]
It is well known that the KPP- fisher equation \eqref{W} 
has a nonnegative solution if and only if $(s+2\ep s)  \ge 2\sqrt{\ep(\ep+1)}s$, which  automatically holds in this case.
More precisely we refer to the following lemma in \cite{Wa}, which is based on the standard argument
for the Fisher equation in \cite{LeNe, Mu} etc.
\begin{lemma}\label{KPPtheorem}[Wang, Lemma 3.2.]
A nonnegative traveling wave solution $W$ of \eqref{W} exists.
It satisfies $W' <0,$ 
 \[ W_- = \frac{(1+\ep)s^2}{N_0}, \mbox{ and } \quad  W_+ = 0.\]
  Moreover it is unique up to
a translation in $x$, or $t$ and has the following asymptotic behavior:
\[ W(x) -  \frac{(\ep+1)s^2}{N_0} \sim Ce^ {\la  x} \mbox{ as } x \to  -\infty \mbox{ and } \quad 
W(x) \sim (e^{-sx}) \mbox{ as } x \to  \infty\]
with   $ \la =\frac{-s(1+ 2\ep) + s \sqrt{ (1+2\ep)^2 + 4\ep(1+\ep)}}{2\ep}$ when $\ep >0$ and $\la = -s$ when $\ep=0$.
\end{lemma}
To be  consistent with  Lemma \ref{KPPtheorem}, the wave speed $s$ is uniquely determined by 
$n_-$,  
\begin{equation} s^2 = \frac {n_-}{1+\ep}.\end{equation}
%What it follows we let $n_- =1$ and $s = 1/\sqrt 2$ without loss of generality. Also we see that 
%\[ W'(-\infty) = 0, \quad W'(\infty)= 0.\] 
Also  from $(C'-sC) =  e^{sz} W' < 0$ and the boundary condition $C(-\infty)= C'(-\infty)=0$, 
we have 
\[ W' (-\infty) =0.\] 
We convert back $N$, $C$ satisfying \eqref{twave}  from $W$ with the required front condition 
\eqref{nfront}, \eqref{cfront} and \eqref{derivative}.  The monotonicity of $N, C$ follows from $W' < 0$;
the profile of $N$ is decreasing, and $C$ is increasing left to rightward.
\begin{figure}
\begin{center}
\includegraphics[scale=0.8]{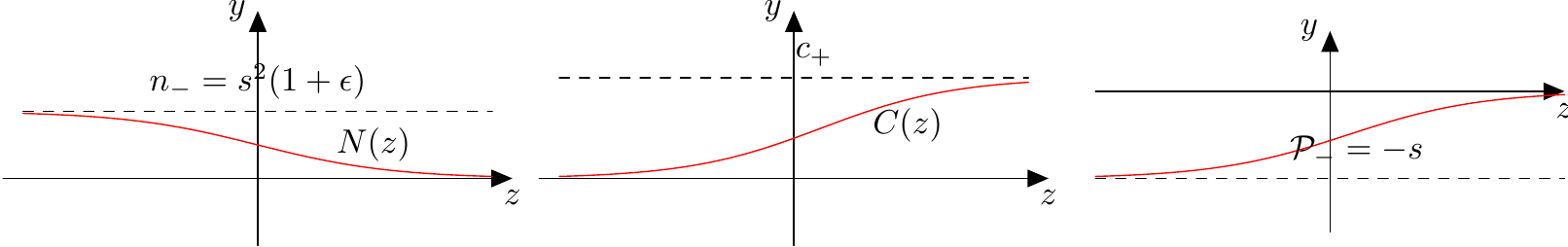}
\caption{ Monotonicity of $N$, $C$ and $\mathcal P$.}
\end{center}
\end{figure}
\indent
In the zero diffusion ($\ep =0$)
we can explicitly solve \eqref{NC} 
\[ sC' = e^{-sz} C^2 N_0\]
such that 
\[ C = \frac{c_+}{ c_+ \frac{N_0}{s^2} e^{-sz} +1}.\]
Then 
\[ N = \frac{c_+ N_0 e^{-sz}}{c_+ \frac{N_0}{s^2} e^{-sz} +1}.\]
We have the relation $ n_- = s^2$  and the monotone profile of $(N, C)$.
\end{subsection}
\begin{subsection}{Cole-Hopf transformation}\label{hopf}

%Using   \eqref{NC}, we find the equation $(u, w)$ satisfy to be 
%\begin{align}\label{per} \begin{aligned}
%&u_t - s u_x - \Del u
%= - \na \cdot \left( \frac{ (C', 0) + \na v }{C+ w} (N + u) - \frac{ (C', 0)}{C} N\right)\\
%&w_t - s w_x - \ep\Del w =  -N w- C u-uw
%\end{aligned}\end{align}
%at $(z = x-st, y, t)$.

By  the Cole-Hopf transformation
\[ q = - \na \ln c = -\frac{\na c}{c}\]
 We translate \eqref{eq:main} into the  divergence type equations of $(n, q)$  
 \begin{align}\label{nq} \begin{aligned}
\pa_t n - \Del n& = \na \cdot (nq), \\
\pa_t q - \ep \Del q & =  -2\ep (q \cdot \na ) q + \na n,
\end{aligned} \end{align}
with $ ((q\cdot \na) q )_i = \sum_{k=1,2} q_k \pa_k q_i$. 
This works only for $m=1$ in \eqref{KS}.
 The Cole-Hopf transformation is initiated in \cite{LiWa} for the stability problem of the  traveling wave equation in the one-dimensional angiogenesis. It  appears in \cite{LiLi} studying the  multi-dimensional  angiogenesis equation for $\ep=0$ with vanishing boundary conditions, too.
 \\
 \indent
 We shall add the condition
 \begin{equation}\label{curl} \na \times q = \pa_2q_1 - \pa_1 q_2 = 0.
 \end{equation}
 This curl free condition on $q$ is preserved in time  until  the classical local solution $(n, q)$ exists. 
It holds that
\begin{equation}\label{na}
  q \cdot \na q = \frac 12 \na |q|^2 
  \end{equation}
  with \eqref{curl}.
We denote   \begin{equation}\label{CP}  P=(\calP,0)= -(C'/C, 0)
\end{equation} %by $P$
 then the perturbation  \eqref{ncp} around $(N, C)$ is written by
\begin{align}\label{up} 
n(x, y, t) = N(x- st) + u (x-st, y, t), 
\quad 
q(x, y, t) = P(x- st) + \na \psi (x-st, y, t).
\end{align} 
%The second equation is converted back to 
%\[ \na \ln c = \na \ln C (x-st)  + \na \psi(x-st, y, t) \]
%which accounts for the unusual form of perturbation in \eqref{ncp}.\\
\indent
Note that  $N, P(:= (\calP, 0)) $ solves
\begin{align}\label{np} \begin{aligned}
- sN'- N''& =  (N\calP)', \\
-s\calP' -  \ep \calP ''& =  -2\ep \calP  \calP'+  N'
\end{aligned} 
\end{align}
%Taking  $\int_{-\infty}^x$ and   $ \int_x^{\infty}$ respectively on the first equation, and 
The boundary conditions of  $(N, P)$ are  inherited from those of $(N, C)$:
\begin{align}\label{np_copy_bd} \begin{aligned}
  & N(-\infty) = (1+\ep) s^2, N(\infty) =0, \quad  \calP(-\infty) = -s , \calP(\infty)=0,\\
  & N'( \pm \infty) = 0, \quad P'(\pm \infty) =0.
  \end{aligned}\end{align}
  For a while we denote
  \[ \na \psi = p.\] So the curl-free condition is endowed  to $p$. 
The perturbation $(u, p)$ satisfies
\begin{align}\label{perturb} \begin{aligned}
&u_t - s u_z - \Del u
=  \na \cdot ( Np + Pu+ up ),\\
&p_t - s p_z - \ep\Del p =  -2\ep \left(  ( (P+p )\cdot \na ) (P+p) - (P\cdot \na) P \right) + \na  u.
\end{aligned}\end{align}

%In the  zero diffusion case,
% we have
%\begin{align}\label{nqzero} \begin{aligned}
%\pa_t n - \Del n& = \na \cdot (nq) \\
%\pa_t q  & =   \na n
%\end{aligned} \end{align} 
%by $\ep =0$.
%The perturbation $(u, p)$ satisfies
%\begin{align}\label{perturbzero} \begin{aligned}
%&u_t - s u_z - \Del u
%=  \na \cdot ( Np + Pu+ up )\\
%&p_t - s p_z =   \na  u.
%\end{aligned}\end{align}
If we let 
 \begin{align*} 
 u = \na \cdot \varphi, \quad  p = \na \psi,
 \end{align*}
the equation \eqref{perturb} is put into the equation on $(\varphi, \psi)$ 
\begin{align}\label{eq0:phipsi_with_ep} \begin{aligned}
\varphi_t - s \varphi_z - \Del \varphi
& =   N \na \psi + P\na \cdot \varphi + \na\cdot \varphi \na \psi, \\
\ps_t - s \ps_z -\ep \Del \ps &=  -2\ep  P \cdot\na \ps -\ep  |\na \ps|^2 + \dv\varphi.
\end{aligned}
\end{align}

\subsection{Main proposition for Theorem \ref{theoremnc}}
%For the main stability theorem \ref{theoremnc},
From now on, we fix a planar traveling wave solution    $(N,C)$ of the  system \eqref{eq:main}
for $\ep=0$
with \eqref{nfront}, \eqref{cfront} and  \eqref{derivative}, obtained in Section \ref{existence_traveling}.
The pertubation $(\vp, \ps)$ satisfies the system 
\begin{align}\label{eq0:phipsi} \begin{aligned}
\varphi_t - s \varphi_z - \Del \varphi
& =   N \na \psi + P\na \cdot \varphi + \na\cdot \varphi \na \psi, \\
\ps_t - s \ps_z  &=     \dv\varphi
\end{aligned}
\end{align}
for $ (t, z, y) \in [0, \infty)\times \bbr \times [0, \la]$ where
$P=-(C'/C,0)$ as in \eqref{CP}.  In addition, recall
 the function spaces and the periodic condition (in $y$-direction) for $(\vp,\ps)$ which are discussed prior to the main Theorem \ref{theoremnc}. \\
 
First 
we state a result on the local existence of solutions for the system \eqref{eq0:phipsi} which can be proved in the standard manner. The proof is presented at the end of the section 3 (Subsection \ref{pf_local_exist}).
\begin{proposition}\label{zerolocal}
For %$\ep\ge 0$,
 $\la>0$ and $M>0$,  there exists $T_0>0$ such that
for any initial data $ (\varphi_0, \psi_0)$ which is $\la-$periodic in $y$ with $ \| \varphi_0\|^2_{H^3_w} + \| \psi_0\|^2_{H^3} + \|\na \psi_0\|^2_{H^2_w}  < M $,
 the system \eqref{eq0:phipsi} has
a unique solution $(\vp, \psi)$ on $[0, T_0]$ which is $\la-$periodic in $y$ 
for each $t\in[0,T_0]$
with $ \varphi|_{t=0} = \varphi_0$ and $\psi|_{t=0} = \psi_0$ and 
\[  \varphi \in L^{\infty}(0, T_0; H^3_w), \quad \psi \in L^{\infty}(0, T_0; H^3), \quad \nabla\psi \in L^{\infty}(0, T_0; H^2_w)  \color{black} .\]
Moreover it holds the following inequality:  %\marginpar{ $\psi \in H^3_w$}
\begin{equation}\label{double}
 \sup_{t\in [0, T_0]}\Big( \| \varphi\|^2_{H^3_w}+ \| \psi \|^2_{H^3 }+\| \na \psi\|^2_{H^2_w}\Big) \le 2 M.
 \end{equation}
\end{proposition}

%\comment{state link to ($n, c)$ inverting Cole-Hopf translformation}
Let us state the result on the existence of global classical solution to the system \eqref{eq0:phipsi}.
Due to our derivation of \eqref{eq0:phipsi} from \eqref{eq:main}, the main Theorem \ref{theoremnc} is the direct consequence of Proposition \ref{zeroglobal}, which will be
proved in Section $3$. 
\begin{proposition}\label{zeroglobal} 
There exist constants %$\ep_0>0,\, 
$m_0>0,\, C_0>0$, and $\la_0>0$ such that
%for $ 0\le \ep < \ep_0$,
  for $0<\la\le \la_0$ and
for any
 initial data $(\vp_0, \ps_0)$ which is $\la-$periodic in $y$ with %  satisfying 
% $\varphi_0 \in H^3_w $ and $\psi_0 \in H^3$ with $\nabla\psi_0\in H^2_w$. Then  such that if
  $M_0:=\| \vp_0\|^2_{H^3_w} + \| \ps_0\|^2_{H^3}  + \|\nabla\psi_0\|^2_{H^2_w}  \color{black}\le m_0$,  
% the length of transversal direction $\la$ is sufficiently small,
 there exists a unique global classical solution $(\vp,\ps)$  of \eqref{eq0:phipsi}
  which is $\la-$periodic in $y$ for each $t>0$ 
with $ \varphi|_{t=0} = \varphi_0$ and $\psi|_{t=0} = \psi_0$.
Moreover, $ (\phi, \psi)$ satisfies the inequality
%  \marginpar{add \\$\| \na \psi \|_{H^2_w}$}
\[   \sup_{t\in [0, \infty)}\Big(\| \varphi\|^2_{H^3_w} + \| \psi\|^2_{H^3} + \| \na \psi \|^2_{H^2_w}\Big) 
%+ \int_0^{\infty} \int  \sum_{1\le k \le 4} \frac{|\na^k \varphi|^2}{N} + \int_0^{\infty} \int \sum_{1\le k\le 3} N|\na^k \psi|^2 + \frac{|\na^k \psi|^2}{N}
+\int_0^{\infty}  \Big(  \| \na  \vp\|_{H^3_w}^2
+   \| \na  \ps\|_{H^2_w}^2
%+\ep     \| \na^{4} \ps\|_w^2
 \Big)dt
 \le C_0M_0.\]

\end{proposition}

In the below we summarize the notations used in the paper.
\[ w(z) : = 1 + e^{sz},\]
\[M(t):=  \sup_{ s\in [0, t]}( \| \varphi(s)\|^2_{H^3_w} 
+ \|\psi(s)\|^2_{H^3} +\|\na \psi(s)\|^2_{H^2_w}),\]
\[ M_0:=  ( \| \varphi_0\|^2_{H^3_w} %+ \| \psi_0\|
+ \|\psi_0\|^2_{H^3}+\|\na \psi_0\|^2_{H^2_w}),  \]
\[  \|  f\|:=\|  f\|_{L^2(\Om)},\]
\[  \|  f\|^2_{k}:=\|  f\|^2_{H^k}=\sum_{|\alpha|=0}^k\int_\Omega|D^\alpha f|^2\,dzdy,\]
\[  \|  f\|^2_{k,w}:=\|  f\|^2_{H^k_w}=\sum_{|\alpha|=0}^k\int_\Omega|D^\alpha f(z,y)|^2\,w(z)dzdy,
\]
\[\int f:=\int_\Omega f(z,y) dzdy,\]
\[\int_0^t g:=\int_0^t g(s) ds.\]

\end{subsection}

\section{Uniform in time estimate when $\ep=0$}\label{sectionuniform}

In this section we introduce the main proposition \ref{uniform_} below which implies
Proposition \ref{zeroglobal}.

\begin{proposition}\label{uniform_}
There exist constants 
%$\ep_0>0,\,
 $\delta_0>0,\, C_0$, and $\la_0>0$ such that
%for $ 0\le \ep < \ep_0$,
% for $0<\la\le \la_0$ and 
 for
  $\la\in(0,\la_0]$, we have the following:\\
 If  $(\varphi, \psi)$ be a local solution of \eqref{eq0:phipsi} on $[0,T]$ for some $T>0$ with $$M(T)\leq\delta\quad \mbox{(the bootstrap assumption)}$$ {for some initial data $(\varphi_0, \psi_0)$ which is $\la$-periodic in $y$}, % from an initial data $(\varphi_0, \psi_0)$. % from the local-existence statement in Proposition \ref{local}.
then   we have %for any $t \in [0, T_0]$
\begin{equation}\label{improve}
  M(T) +  \int_0^{T}  \sum_{l = 1}^4 \| \na^{l} \vp\|_w^2
+\int_0^{T}   \sum_{l = 1}^3\| \na^l \ps\|_w^2 \le C_0M_0.\end{equation}
\end{proposition}
 This proposition will be proved in   
  Subsection \ref{mainproof}.
  \begin{remark}
   
Note that the constant $C_0$ in the above proposition does not depend on the size of $T>0$. This fact immediately   improves the bootstrap assumption 
 if we take a sufficiently small initial data, for instance, satisfying $C_0M_0<\delta_0/2$.
 Once we prove the proposition, then the global-existence statement in Proposition \ref{zeroglobal} follows the standard continuation argument as shown below: % Proposition \ref{uniform_}:
\end{remark}
\begin{proof}[Proof of Proposition \ref{zeroglobal} from Proposition \ref{uniform_}]
Let $M:=\delta_0/2, \, m_0:=M/C_0$ where $\delta_0$ and $C_0$ are the constants in  Proposition \ref{uniform_}.
We may assume $C_0\ge 1$.  Consider the initial data $(\vp_0, \ps_0)$ with $M(0)=M_0\leq m_0\leq M$. By using this constant $M$ to the local-existence result (Proposition \ref{zerolocal}),
% From the constant $M$ with $M_0\leq M$,
  there exist  $T_0>0$
and the  unique local solution $(\vp, \ps)$ on $[0,T_0]$ with $M(T_0)\leq 2M$. Due to
$M(T_0)\leq2M\leq \delta_0$, we can use the result of Proposition \ref{uniform_} to obtain $M(T_0)\leq C_0 M_0$, which implies $M(T_0)\leq C_0 m_0\leq M$. Hence we can extend the solution up to the time $2T_0$ by Proposition \ref{zerolocal} and we obtain
  $M(2T_0) \leq 2M\leq \delta_0$. Due to Proposition \ref{uniform_}, it implies $M(2T_0)\leq C_0M_0\leq M$. Thus we can repeat this process of extension to get
  $M(kT_0)\leq C_0M_0$ for any $k\in\mathbb{N}$. 
\end{proof}
First we summarize some properties of traveling waves $(N, \mathcal P)$ for later analysis. 
%%%%%%%%%%%%%%%%%%%%%%%
\begin{subsection}{Properties on $(N, \mathcal P)$}
%In the zero diffusion case
From the equations:
\begin{align}\label{np_copy} \begin{aligned}
- sN'- N''& =  (N\calP)', \\
-s\calP' & =   N',
\end{aligned} 
\end{align}
  we observe the relations
 \begin{equation}\label{relation}
 \fn{\calP} = -\frac{1}{s} \quad \mbox{and }  \left( \fn{1} \right)^{'' }= s\left( \fn{1} \right)^{'}
 \end{equation}
%where $P = (\calP, 0)$.
from \eqref{formula:NP} where $s>0$, which is fixed throughout this paper, is the speed of our planar traveling wave $(N,C)$.
 In fact, $(N, P)$ has the explicit formula: % $P = - ( \frac {e^{-sz}N_0}{s} C, 0)$,
\begin{equation}\label{formula:NP}  N = \frac{c_+ N_0 e^{-sz}}{c_+ \frac{N_0}{s^2} e^{-sz} +1}, \quad P = - \left (  \frac{c_+ N_0e^{-sz}}{ c_+ \frac{N_0}{s} e^{-sz} +s}, 0 \right).
\end{equation}
Note that 
%\begin{comment} s\end{comment}
%$s>0$ is a fixed constant,  and 
$N$ and $\calP$ are monotone with
the boundary condition \eqref{np_copy_bd}. %$\calP(-\infty) = -s$, $\calP(\infty) = 0$. 
%\comment{Zero diffusion case. Fill in.} \\
%For non-zero chemical diffusion case, 
%recall the equation \eqref{np_copy} (or \eqref{np}) and boundary conditions \eqref{np_copy_bd}:
%\begin{align}\label{np_copy} \begin{aligned}
%- sN'- N''& =  (N\calP)' \\
%-s\calP' -  \ep \calP ''& =  -2\ep \calP  \calP'+  N'
%\end{aligned} 
%\end{align}
%Moreover we have 
%\begin{align}\label{np_copy_bd} \begin{aligned} & N(-\infty) = %(1+\ep)
% s^2, \quad N(\infty) =0,   \quad  \calP(-\infty) = -s , \quad  \calP(\infty)=0, \\
%  &
%  N'( \pm \infty) = 0, \quad P'(\pm \infty) =0.\end{aligned} 
%\end{align}
%Let me denote $N_{\ep=0}, \calP_{\ep=0}$ be the solution with $\ep=0$ and   $N_\ep, \calP_\ep$ be the solution with $\ep>0$.\\
%Note that $(N_{\ep=0})'=-sN_{\ep=0}(N_{\ep=0}+1)$, $\calP_{\ep=0}=-(1/s)N_{\ep=0}$, $
%   \left( 1/N_{\ep=0}\right)^{'' }= s\left( 1/N_{\ep=0} \right)^{'}$ and
    
%$$N_{\ep=0} = \frac{c_+ N_0 e^{-sz}}{c_+ \frac{N_0}{s^2} e^{-sz} +1}, \quad P_{\ep=0}=(\calP_{\ep=0},0) = - \left (  \frac{c_+ N_0e^{-sz}}{ c_+ \frac{N_0}{s} e^{-sz} +s}, 0 \right).$$
 
%Moreover, we have the following lemma.
In the following lemma, we summarize some properties of $(N,\calP)$.
% which can be proved directly from the explicit formula \eqref{formula:NP}. 
%As a side remark, this lemma holds for the solution $(N, C)$ of the system \eqref{NC} even for $\ep>0$ while the relations \eqref{relation} work only for the case $\ep=0$. . %When $\ep=0$, \eqref{relation} holds from the explicit formula of \eqref{formula:NP}.
\begin{lemma}\label{wave_prop}
%For $s>0$, t
%For the f $(N,\calP)$ of \eqref{np_copy} with $s>0$,
There exists a constant  %$\ep_0>0$, %$0<c_1<c_2$,
%$0<\delta<1$,
%and 
$M>0$ such that 
%for any $\ep\in[0,\ep_0]$, 
we have:\\
 $$ \frac{w}{M}\leq\frac{1}{N }\leq Mw,$$
% $$ 
% %(1-\delta)\frac{1}{s}\leq
% \Big|\frac{\calP_\ep}{N_\ep}\Big|\leq (1+\delta)\frac{1}{s}$$
% $$ 0<N_\ep\leq(1+\delta)(1+\ep_0)s^2$$
 %$$-(1+\delta)s\leq \calP_\ep<0$$
 %$$ (\frac{1}{N_\ep})'\leq s\Big(2(1-\delta_2)\frac{1}{N_\ep}\Big)$$
% $$ \Big(\frac{1}{N_\ep}\Big)'\geq0$$
  %$$|\calP'_\ep|<(1+\delta)(1+\epsilon_0)s^2$$
$$ |{N^{ (k)}}| \leq M, \quad   |{\calP^{(k)}}|  \leq M,\quad \mbox{ for } 0\leq k\le 2,\quad \mbox{ and } $$
$$
\frac{N'}{N}=-(\calP+s),   \quad 
\Big|(\frac{1}{N})'\Big|+\Big|(\frac{1}{N})''
\Big|\leq \frac{M}{N},\quad 
\Big|(\frac{1}{\sqrt {N}})'\Big|\leq \frac{M}{\sqrt {N}}.$$
($N^{(k)}$ is any $k-$th derivative of $N$)
\end{lemma}
\begin{proof}
The first inequality follows from Lemma \ref{KPPtheorem}. The others are easily obtained from the equation \eqref{np_copy} and \eqref{relation}.
\end{proof}
\end{subsection}
%\newpage
\begin{subsection}{Proof of Proposition \ref{uniform_}}\label{mainproof}
 
%We present a series of lemmas by doing energy estimates.

{
First assume $0<\delta_0\leq
%\delta_0\leq
 1$. During the proof, we will take $\delta_0$ as small as needed to have the weighted energy estimate \eqref{improve} (e.g. Lemma \ref{lemma1_2}).}
Let us remind the system \eqref{eq0:phipsi}
\begin{align}\label{eq:phipsi} \begin{aligned}
\varphi_t - s \varphi_z - \Del \varphi
 &=   N
 %_{\ep}
  \na \psi + P
  %_{\ep}
  \na \cdot \varphi + \na\cdot \varphi \na \psi, \\
\ps_t - s \ps_z  &=   \dv\varphi \\ 
\end{aligned}
\end{align} for $ (t, z, y) \in [0,T]\times \bbr \times [0, \la]$. %From now on, we drop the $\ep$ index in $N_{\ep}$ and $P_{\ep}$ for simple presentation. 
We will collect a few lemmas on energy estimates before proving Proposition \ref{uniform_}. Here the constants $C$ which will appear in the following lemmas are independent of $T>0$. %From now on, the time variable $t$ lies on the interval $[0,T]$. 
\begin{lemma}\label{lemma0_}
If  $\la>0$ %and $\ep\ge 0$
 is sufficiently small, we have the following: %for any $t\in[0,T]$,
\begin{equation}\label{eq_lemma0_}
\| \psi \|^2 + \| \varphi\|_w^2 + \int_0^t \| \na \varphi\|_w^2
%+{\ep}\int_0^t \|\na \psi\|^2
\le C ( \| \psi_0\|^2 + \| \varphi_0\|_w^2 )+ C {{M(t)}} \int_0^t 
 \|\na \psi\|_w^2.
%\int_\mathbb{R} \fn{|\na \psi|^2}.
\end{equation}
\end{lemma}
\begin{remark}
This estimate \eqref{eq_lemma0_} is not closed due to the second term in the right-hand side. This obstacle will be studied sequentially in  the next lemmas. 

\end{remark}
\begin{proof}

Multiply  $\fn{\varphi}$ to the $\varphi$ equation and $\psi$ to the $\psi$ equation. 
Integrating by parts, we have
\begin{align*}
&\frac 12 \ddt \left( \int \fn{ |\varphi|^2} +\int  |\psi|^2 \right) 
+ \int\fn{ \sum_{i} |\na \varphi^i|^2} %+ \ep\int |\na \psi|^2
+ 
\frac s 2 \int |\varphi|^2 \left( \fn{1}\right)'\\
&= \frac 12 \int |\varphi|^2 \left(  \fn{1}\right)''+
\int \fn{\calP} \varphi^1 \na \cdot \varphi  
+  \int \fn{\varphi \cdot \na\psi} \na\cdot \varphi.
%-%\underbrace{ 2\ep \int P\cdot \na\ps \psi }_{= 
% 2\ep \int \calP \ps_z \psi
% %}
%- \ep \int |\na \ps|^2 \ps.
\end{align*}
We estimate the cubic term first:
\begin{align*}
  \int \fn{\varphi \cdot \na\psi} \na\cdot \varphi & \le  { C\int \fn{\| \vp\|_{L^{\infty}}^2| \na \psi|^2} 
 + \frac{1}{4} \int \fn{|\na \vp|^2}  \le C  {M(t)} \int \fn{| \na \psi|^2}  + \frac{1}{4}\int \fn{|\na \vp|^2}
 }
\end{align*}
 where we used $\|\vp\|_{L^\infty}\leq C \|\vp\|_{H^2}\leq C \sqrt{M(t)}$ by the Sobolev embedding. % and $M(t)\leq M(T)\leq 1$.\\

%To cover   the quadratic term $\int \fn{\calP} \varphi^1\na \cdot \varphi$,
%We split into
%two cases: $\ep=0$ and $ \ep>0$.

% Without diffusion case($\ep=0$),
 
Thanks to the relations \eqref{relation}, the quadratic term becomes
\begin{align}\label{cancelation}
 \int \fn{\calP} \varphi^1\na \cdot \varphi  & =
-\frac 1s \int   \varphi^1(\pa_z\vp^1+\pa_y\vp^2)
%=-\frac 1s \int    \pa_z((\vp^1)^2)+(\vp^1- \overline{\vp}^1)\pa_y \vp^2  
=-\frac 1s \int     (\vp^1- \overline{\vp }^1)\pa_y \vp^2  
  \end{align} where  $\overline{\vp}(z)=(1/\la)\int_{[0,\la]}\vp(z,y)dy$ is the average in $y$
  of $\vp$ for each $z$. Note that we used the periodic condition for $\vp^2$. As a result, we get  
  \begin{align*}
 \Big|\int \fn{\calP} \varphi^1\na \cdot \varphi\Big|  & 
  \le C\frac {\la}{s} \| \pa_y\vp^1\|_2 \| \pa_y\vp^2\|_2 \le C \la   \| \na\vp\|^2_w\\
  \end{align*}    where we used the Poincar\'{e} inequality 
  \begin{equation}\label{ineq:poincare2_periodic} 
\| \vp(z,\cdot_y)-\overline{\vp}(z) \|_{L^2_y([0,\la])} \le C\la \| \pa_{y} \vp(z,\cdot_y)\|_{L^2_y([0,\la])} \quad\mbox{ for } z\in\mathbb{R}.
\end{equation} Also we have
$$ {\frac s 2 \int |\varphi|^2 \left( \fn{1}\right)'- \frac 12 \int |\varphi|^2 \left(  \fn{1}\right)''} { = 0}$$ thanks to the relations \eqref{relation}.
By assuming smallness of $\la$,   we get the following: 
\begin{align*}
&\frac 12 \ddt \left( \int \fn{ |\varphi|^2} +\int  |\psi|^2 \right) +\frac 18 \int\fn{   |\na \varphi|^2}   
 \leq   C {{M(t)}} \int \fn{| \na \psi|^2}.
\end{align*} It proves Lemma \ref{lemma0_}.\\ % for $\ep=0$.\\

\end{proof}

%%%%%%%%%%%%%%%%%%%%%%%%%%%%%%%%%%%%%%555
%first order
%%%%%%%%%%%%%%%%%%%%%%%%%%%%%%%55
Let's study the first order estimate:
\begin{lemma}\label{lemma1_0} If    $\la>0$  is small enough, then
we have 
\begin{equation}\begin{split}\label{ineq:lem1_}
&\| \na \vp\|_w^2  + \| \na \ps\|^2+ \int_0^t \| \na^2 \vp\|_w^2 \\
%+ \ep\int_0^t \|\na^2 \psi\|^2\\
&\le C ( \| \na \vp_0\|_w^2 + \| \na \ps_0\|^2+  \| \psi_0\|^2 + \| \varphi_0\|_w^2 ) 
% \phantom{XXXXX}
+  C \int_0^t \int {N} |\na \ps|^2 + C {{M(t)} }\int_0^t \int \fn{\abs{\na \ps}}.
\end{split}\end{equation}\end{lemma}
\begin{remark}
%\begin{enumerate}
%\item 
This estimate \eqref{ineq:lem1_} is not closed yet due to the last two terms. Later in  Lemma \ref{lemma01_} (see the estimate \eqref{ineq:lem1_closed}), it will be closed by Lemma \ref{lemma1_1} and Lemma \ref{lemma1_2}. Indeed, the first one $\int_0^t \int {N} |\na \ps|^2 $ among these two  will be covered by  the second one $\sqrt{M(t)} \int_0^t \int \fn{\abs{\na \ps}}$ (see Lemma \ref{lemma1_1}) while the second term will be absorbed into the dissipation term 
 $\int_0^t \| \na^2 \vp\|_w^2 $ on the left-hand side of \eqref{ineq:lem1_} (see Lemma \ref{lemma1_2}).
%\end{enumerate}
\end{remark}
\begin{proof}
Differentiating in $z$, we have
\begin{align*}
\vp_{tz} - s\vp_{zz} - \Del \vp_z & = N' \na\psi + N \na \psi_z + P' \dvg\vp + P \dvg \vp_z + \na \psi_z \dvg \vp + \na\psi \dvg \vp_z,\\
\psi_{tz} -s\psi_{zz} 
%-\ep \Del \ps_z 
& =
% -2\ep  (P \cdot\na \ps)_z -\ep  (|\na \ps|^2)_z  +
  \dvg \vp_z.
%\ps_t - s \ps_x -\ep \Del \ps &=  -2\ep  P \na \ps -\ep  |\na \ps|^2 + \dv\varphi
\end{align*}
Multiply  $\fn{\varphi_z}$ and $\psi_z$ to the above equation. 
Integrating by parts, we have
\begin{align*}
 \frac 12 \ddt \lr{ \int \fn {\abs{\vp_z} } + \abs{\ps_z}} + \int \sum_{ij}\fn{ \abs{\pa_j \vp^i_{z}}} & = %\underbrace
 { \frac 12 \int \abs{\vp_z} \left( \fn{1} \right)''- \frac s2 \int \abs{\vp_z} \Fn' }
% _{\leq  \frac{1}{2s}\int|\vp_{zz}|^2(\frac{1}{N})'\leq \frac 12 \int\frac{|\vp_{zz}|^2}{N}}
 \\
 & +\int \fn{N'} \na \ps \vp_z  +\int \fn{P'} \dv \vp \vp_z + \int \fn{P} \dv \vp_z \vp_z\\
 &+  \int \na \ps_z \dv \vp \fn{\vp_z} + \int \na\ps \dv \vp_z \fn{\vp_z}.\\
%& -\underbrace{ 2\ep \int (P\cdot \na\ps)_z \psi_z }_
%{=  2\ep \int (\calP \ps_z)_z \psi_z}
%- \ep \int (|\na \ps|^2)_z \ps_z.
\end{align*}
Similarly,  %for
%\begin{align*}
%\vp_{ty} - s\vp_{zy} - \Del \vp_y & = N \na \psi_y + P \dvg \vp_y + \na \psi_y \dvg \vp + \na\psi \dvg \vp_y\\
%\psi_{ty} -s\psi_{zy} -\ep \Del \ps_y & = -2\ep  (P\cdot \na \ps)_y -\ep  (|\na \ps|^2)_y  + \dvg \vp_y\\
%\end{align*}
we get
\begin{align*}
\frac 12 \ddt \lr{ \int \fn {\abs{\vp_y} } + \abs{\ps_y}} +\int \sum_{ij}\fn{ \abs{\pa_j \vp^i_{y}}}  
   &=  
   %\underbrace
   { \frac 12 \int \abs{\vp_y} \left( \fn{1} \right)'' - \frac s2 \int \abs{\vp_y} \Fn' }
   %_{\leq  \frac{1}{2s}\int|\vp_{yz}|^2(\frac{1}{N})'\leq \frac 12 \int\frac{|\vp_{yz}|^2}{N}} 
 \\
   &+\int \fn{P} \dv \vp_y \vp_y \\
   &+ \int \na \ps_y \dv \vp \fn{\vp_y} + \int \na\ps \dv \vp_y \fn{\vp_y}.\\
%  & -\underbrace{ 2\ep \int (P\cdot \na\ps)_y \psi_y }_
%{=  2\ep \int \calP \ps_{zy} \psi_y}
%- \ep \int (|\na \ps|^2)_y \ps_y.
\end{align*}

First, we estimate by using Lemma \ref{wave_prop}
\begin{align*}
&\frac 12 \int |\na\vp|^2 \lr{ \fn{1}}''-\frac s 2 \int |\na\varphi|^2 \left( \fn{1}\right)'
\le -\int \na\vp\cdot\na\vp_z  \left( \fn{1}\right)'+\frac s 2 \int |\na\varphi|^2 \left|\left( \fn{1}\right)'\right|\\
&\le  C\int |\na\vp||\na\vp_z |  \left( \fn{1}\right) +\frac {Cs}{2} \int |\na\varphi|^2  \left( \fn{1}\right) 
\le  \frac{1}{4}\int  |\na\vp_z |^2\left( \fn{1}\right)+(C+\frac {Cs}{2} ) \int |\na\varphi|^2  \left( \fn{1}\right).
 \end{align*}
We estimate the quadratic terms as follows:
\begin{align*}
\int \fn{P} \dv \vp_z \vp_z+\int \fn{P} \dv \vp_y \vp_y  \le
 \| P\|_{L^{\infty}} (\|  \frac{\na\vp_z}{ \sqrt{N}} \|\| \frac{\vp_z}{\sqrt{N}}\| 
+\|    \frac{\na\vp_y}{ \sqrt{N}} \|\| \frac{\vp_y}{\sqrt{N}}\| )
 \le
\frac{1}{4}\|  \frac{\na^2\vp}{ \sqrt{N}} \|^2+C\| \frac{\na \vp}{\sqrt{N}}\|^2,
\end{align*}
\begin{align*}
\int \fn{N'} \na \ps \vp_z & \le \| \fn{N'}\|_{L^{\infty}}\| {\sqrt N}{\na\ps}\| \| \frac{\vp_z}{ \sqrt{N}} \|
\le C\| {\sqrt N}{\na\ps}\|^2+ C \| \frac{\na\vp}{ \sqrt{N}} \|^2,\\
\int \fn{P'} \dv \vp \vp_z  &\le 
\| P'\|_{L^{\infty}} \| \frac{\na \vp}{\sqrt N} \|^2
\le 
C\|\frac{\na \vp}{\sqrt N}  \|^2.
\end{align*}

The cubic terms are bounded by  
\begin{align*}
\int  | \na\ps|| \na\vp| \fn{|\na^2\vp|} 
+\int  | \na\ps|| \na\vp|^2 \underbrace{|({1}/{N})'|}_{\leq C/{N}}
%+ \int |\na \ps \dv \vp \fn{\vp_{zz}}| 
%+ \int | \na \ps \dv \vp \vp_z \frac{N'}{N^2} | \\
%& \le C \sqrt{M(t)} ( \| \na^2 \vp/\sqrt{N} \|^2  + \| \na \vp/\sqrt{N} \|^2)\\
& {\le \frac{1}{4}  \| \na^2 \vp/\sqrt{N} \|^2 + C(M(t)+\sqrt{M(t)})\| \na \vp /\sqrt{N}\|^2}\\
& { \le \frac{1}{4}  \| \na^2 \vp/\sqrt{N} \|^2 + C\| \na \vp /\sqrt{N}\|^2}
\end{align*}
by  $ \| \na\ps\|_{L^{\infty}} \le C\sqrt{M(t)}$ and {by the assumption $M(t)\leq M(T)\leq 1$.}\\

 \iffalse
%%%%%%%%%%%%%%%%%%
%\ep estimate begin
%%%%%%%%%%%55
For the $\ep$ terms, %\begin{align}
%& -2\ep \int (\calP \ps_z)_z \psi_z- \ep \int (|\na \ps|^2)_z \ps_z  - 2\ep \int \calP \ps_{zy} \psi_y- \ep \int (|\na \ps|^2)_y \ps_y,
%\end{align}
 we estimate
\begin{align*}
&-2\ep \int (\calP \ps_z)_z \psi_z  - 2\ep \int \calP \ps_{zy} \psi_y=
  -2\ep \Big(\int \calP' \ps_z \psi_z
-2\ep \int \calP \ps_{zz} \psi_z 
  - 2\ep \int \calP \ps_{zy} \psi_y \Big)\\
  &\leq C\ep \|P\|_{L^\infty}\|\na\ps\|\|\na^2\ps\|+C\ep \|P'\|_{L^\infty}\|\na\ps\|^2  \\
    &\leq C\ep  \|\na\ps\|^2+\frac{\ep}{4}\|\na^2\ps\|^2   
\end{align*}
and

\begin{align*}
&- \ep \int (|\na \ps|^2)_z \ps_z  
- \ep \int (|\na \ps|^2)_y \ps_y=
- \ep \int (|\na \ps|^2) \ps_{zz}  
- \ep \int (|\na \ps|^2) \ps_{yy}\\
&\leq  C\ep M(t) \int |\na \ps||\na^2 \ps|
\leq  C\ep \|\na \ps\|^2+\frac{\ep}{4}\|\na^2 \ps\|^2.
\end{align*}
\fi
Adding up all the estimates above, we have 
\begin{align*}
&\frac 12 \ddt \lr { \int \fn{\abs{\na \vp}} + \int \abs{\na \ps} } + \frac 14\int   \fn{\abs{ \na^2\vp } }
%+ \frac{\ep}{4}\int |\na^2 \psi|^2
\le  
C\|{\sqrt{N}} {\na \ps}\|^2  + C   \| \frac{\na \vp}{\sqrt{N}}\|^2.
% +C\ep \|\na \ps\|^2.
\end{align*}
After integration in time, thanks to Lemma \ref{lemma0_}, we can control the last   term  above so that we arrive at \eqref{ineq:lem1_}.
\end{proof}

\begin{lemma}\label{lemma1_1} If   $\la>0$
%, and $\ep\geq0$
 is small enough, then
we have 
 \begin{align} \label{claim1_lemma1_}
& \int_0^t \int N |\na \ps|^2 
%+\ep\int_0^t \int  |\na^2 \ps|^2 
\le C \lr{ \| \na \ps_0 \|^2 + \| \ps _0\|^2 + \| \vp_0\|_w^2} 
+ C\sqrt{M(t)} \int_0^t  \int \fn{ |\na \ps |^2}.
\end{align}
\end{lemma}
\begin{remark}
\begin{enumerate}
\item  Roughly speaking, this estimate moves $N(z)$  to the denominator, while the  integral multiplied by
$\sqrt{M(t)}$ which can be assumed small enough.  
\item  This estimate for $ \int_0^t \int N |\na \ps|^2$ is due to the special structure 
\eqref{eq0:phipsi}. Note that  the $\vp$-equation has the term $N\na\ps$ on its right hand. %For this reason, we do not expect a similar estimate for its lower order part $ \int_0^t \int N | \ps|^2$.
\end{enumerate}
\end{remark}
\begin{proof}

Multiplying $\na\ps$ to the $\vp$-equation, %: $\varphi_t - s \varphi_z - \Del \vp =   N \na \psi + P\na \cdot \varphi + \na\cdot \varphi \na \psi$, 
we have 
\begin{equation}\label{eq_lemma0__first_claim_intermediate}
 N |\na \ps|^2 = \underbrace{\vp_t \cdot \na \ps}_{(A)} - s\vp_z \cdot\na \ps -\underbrace{\Del \vp \cdot\na \ps}_{(B)} 
%- \calP \dv \vp \na \ps 
-(\dv \vp)P\cdot\na\ps
- \dv \vp | \na \ps|^2.
\end{equation}
Let us rewrite the terms $(A)$ and $(B)$.  
\begin{align*}
(A)%=\vp_t \cdot\na \ps 
& = (\vp \cdot\na \ps)_t - \vp \cdot\na \ps_t \\&=
  (\vp \cdot\na \ps)_t - \vp\cdot(s \na \ps_z + \na (\dv \vp)
 % +\ep \Del \na\ps  -2\ep \na( P \cdot\na \ps) -\ep  \na(|\na \ps|^2) )
\end{align*} and 
\[(B) 
% \Del \vp \cdot \na \ps 
 = (\dv \vp)_z \ps_z + ( \vp^1_{yy} - \vp^2_{zy}) \ps_z + ( \dv \vp)_y \ps_y 
+ (\vp^2_{zz} - \vp^1_{zy}) \ps_y.\]
%Recall the $\ps$ equation,
%\[  \ps_t - s \ps_z -\ep \Del \ps =  -2\ep  P \cdot\na \ps -\ep  |\na \ps|^2 + \dv\varphi.\]
Integration by parts gives
\[ \int(B)
%=\int \Del \vp \cdot \na \ps
 =  \int  (\dv \vp)_z \ps_z + ( \dv \vp)_y \ps_y
  =  \int  \na(\dv \vp) \cdot \na\ps.\]
 By using the $\ps$-equation which has $\na\cdot\vp$ on its right side,
we get
\begin{align*} \int \na (\dv \vp)\cdot\na \ps &= \int \na (\ps_{t} - s \ps_{z}
% -\ep \Del \ps +2\ep  P \cdot\na \ps+\ep  |\na \ps|^2
)\cdot \na  \ps  = \frac 12 \ddt \int |\na\ps|^2.
%+\int \na( 
% -\ep \Del \ps +2\ep  P \cdot\na \ps+\ep  |\na \ps|^2
%)\cdot\na \ps.
\end{align*}
%\[ \int  (\dv \vp)_z \ps_z = \int (\ps_{t} - s \ps_{z}
% -\ep \Del \ps +2\ep  P \cdot\na \ps+\ep  |\na \ps|^2
%)_z \ps_z = \frac 12 \ddt \int \ps_z^2
%+\int ( 
% -\ep \Del \ps +2\ep  P \cdot\na \ps+\ep  |\na \ps|^2
%)_z \ps_z
%,\]
%\[ \int  (\dv \vp)_y \ps_y = \int (\ps_{t} - s \ps_{z}
%-\ep \Del \ps +2\ep  P \cdot\na \ps+\ep  |\na \ps|^2
%)_y \ps_y = \frac 12 \ddt \int \ps_y^2
%+\int ( 
% -\ep \Del \ps +2\ep  P \cdot\na \ps+\ep  |\na \ps|^2
%)_y \ps_y.\]
Integration on \eqref{eq_lemma0__first_claim_intermediate} gives us
\begin{align*}%\label{ineq_claim1_lemma1}
\int N |\na \ps|^2 =& 
\ddt \int \vp \cdot\na \ps - \int s \vp \cdot \na \ps_z -\int\vp\cdot (\na (\dv \vp)
%- \int\vp\cdot
%+ \ep \Del \na\ps  -2\ep \na( P \cdot\na \ps) -\ep  \na(|\na \ps|^2) 
)\\ \nonumber
&- \int s \vp_z\cdot \na \ps \\ \nonumber
&- \frac 12 \ddt\int( \ps_z^2 + \ps_y^2)\\ \nonumber
%+\int \na( 
% \ep \Del \ps -2\ep  P \cdot\na \ps-\ep  |\na \ps|^2
%)\cdot\na\ps
&- \int (\dv \vp)P\cdot\na\ps   - \int  \dv \vp | \na \ps|^2. 
  \end{align*}
 We rearrange the above to get
 \begin{align}\label{ineq_claim1_lemma1}
\int N |\na \ps|^2  
=& \ddt \int \vp \cdot\na \ps - \frac 12 \ddt \int  |\na\ps|^2 \\& \nonumber + \int | \dv \vp|^2 - \int (\dv \vp)P\cdot\na\ps - \int  \dv \vp | \na \ps|^2\\ \nonumber
%&+\ep\int \na( 
%  \Del \ps -2  P \cdot\na \ps-   |\na \ps|^2
%)\cdot\na\ps - \ep\int\vp\cdot(   \Del \na\ps  -2 \na( P \cdot\na \ps) -  \na(|\na \ps|%^2) )\\& \nonumber
=(I)+(II).%+(III).
 \end{align}
 Note that 
 \begin{align*}
 {\int_0^t(I)}=\int_0^t\lr{\ddt \int \vp \cdot\na \ps - \frac 12 \ddt \int  |\na\ps|^2}\leq {C(\|\vp(t)\|^2 +  \|\na\ps_0\|^2+\|\vp_0\|^2)}&
 \end{align*} by 
$ \int  \vp \na\ps \le C\| \vp \| ^2 + \frac 12   \| \na \ps\|^2$.

For the second term $(II)$, we estimate 
 \begin{align*}
  \int | \dv \vp|^2-\int (\dv \vp)P\cdot\na\ps 
 %-2s \int \vp_z \na \ps  
 & \le  C\int \fn{ | \na \vp|^2} + \frac 14 \int  N |\na \ps|^2, \\
  \int  |\dv \vp| | \na \ps|^2 & \le C\sqrt{M(t)}    \int {| \na \ps|^2} \le C\sqrt{M(t)}    \int \fn{| \na \ps|^2}
 \end{align*}
 by bounding $\| \dv\vp \|_{L^{\infty}}  \le C \| \dv\vp \|_{H^2}\le C \sqrt{M(t)}$. 
% $\| \na \ps \|_{L^{\infty}} \le C M(t)$.

\iffalse

%%%%%%%%%%%%%%%%%%
%\ep estimate begin
%%%%%%%%%%%55
For the $\ep$-term $(III)$,   we estimate \begin{align*}
 \ep \int \na(    \Del \ps -2   P \cdot\na \ps-   |\na \ps|^2)\cdot\na\ps &=
-\ep\int |\na^2\ps|^2
-\ep \int \na(     2   P \cdot\na \ps+  |\na \ps|^2)\cdot\na\ps\\
&\leq-\ep\|\na^2\ps\|^2 +C\ep  \|\na\ps\|^2+\frac{\ep}{4}\|\na^2\ps\|^2\\
%&\leq-\frac{3\ep}{4}\|\na^2\ps\|^2 +C\ep  \|\na\ps\|^2 
\end{align*} and 
\begin{align*}
 - \ep\int\vp\cdot(   \Del \na\ps  -2 \na( P \cdot\na \ps) -  \na(|\na \ps|^2) ) 
&\leq  C\ep\int|\na\vp||\na^2\ps|  + |\na\vp||\na \ps| +|\na\vp||\na \ps|^2\\
&\leq \frac\ep 4\|\na^2\ps\|^2 +C\ep  \|\na\vp\|^2+C\ep  \|\na\ps\|^2+C\ep M(t)  \|\na\ps\|^2\\
&\leq \frac\ep 4\|\na^2\ps\|^2 +C  \|\na\vp\|_w^2+C\ep  \|\na\ps\|^2.
\end{align*}
\fi

Now integrating \eqref{ineq_claim1_lemma1} in time, we get
 \begin{align*} 
\int_0^t\int& N |\na \ps|^2  
%=& \int_0^t (I)+\int_0^t(II)+\int_0^t(III)\\
\leq 
C(\|\vp(t)\|^2 + \|\na\ps_0\|^2+\|\vp_0\|^2)\\
& +\int_0^t\Big(C  \|\na\vp\|_w^2+ \frac 14 \int  N |\na \ps|^2  +C\sqrt{M(t)}    \int \fn{| \na \ps|^2} %- \frac\ep 2\|\na^2\ps\|^2 +C\ep  \|\na\ps\|^2
\Big).
 \end{align*}

%%%%%%%%%%%%%%%%%%
%\ep estimate end
%%%%%%%%%%%55

By Lemma \ref{lemma0_} and {by $M(t) \leq \sqrt{M(t)}$,} we have 
% \begin{align*} 
% \int_0^t \int N |\na \ps|^2 
% +\ep\int_0^t \int |\na^2 \ps|^2 
% \le& C \lr{ \| \na \ps_0 \|^2 + \| \ps _0\|^2 + \| \vp_0\|_w^2} 
%+ C\sqrt{M(t)} \int_0^t  \int \fn{ |\na \ps |^2}.
% \end{align*}
%This proves the claim 
\eqref{claim1_lemma1_}.  
\end{proof}
What remains to close \eqref{ineq:lem1_} is to control the term
 $\sqrt{M(t)}\int _0^ t \int \fn {|\na \ps|^2 }$. Note that $\psi$-equation lacks  diffusion of closing the estimates without weight. However, our choice of the one-sided exponential weight function $w(\cdot)$ gives a one-sided dissipation $\int_0^t\int_{z>0}\frac{|\nabla \psi|^2}{N}$ as shown below.
\begin{lemma}\label{lemma1_2} If $M(T)>0$ and $\la>0$
%, and $\ep\geq0$ 
are small enough, then
we have 
 \begin{align} 
& \int \fn {|\na \ps|^2} +   \int _0^ t \int \fn {|\na \ps|^2 }
%+ \ep \int_0^t\int    \fn{|\na^2\ps|^2}
\le  C(  \| \na \ps_0 \|_w^2 + \| \ps _0\|^2 + \| \vp_0\|_{w}^2) + C\int_0^t \int \fn{ |\na^2\vp|^2}. \label{claim2_lemma1_}
\end{align}
\end{lemma}
 
\begin{proof}
First we take $\na$ to the $\ps$-equation then multiply by  $ w \na \ps$ to get
%Multiplying $ w \na \ps$ to the equation $$ \na \ps_t - s \na\ps_z -\ep \Del \na\ps = \na (\dv \vp) -2\ep \na( P \cdot\na \ps) -\ep  \na(|\na \ps|^2), $$ we have
\begin{align*} 
&\frac 12 (w |\na \ps|^2)_t - \frac s2 ( w |\na \ps|^2)_z + \frac s2 w' |\na \ps|^2  
 = w \na (\dv \vp) \cdot \na \ps. 
%+\underbrace{\ep w \na \ps\cdot\lr{\Del \na\ps -2  \na( P \cdot\na \ps) -   \na(|\na \ps|^2)}}_{\ep\mbox{-terms}}.
\end{align*}
 
Note that for some $c>0$
\begin{align}
w' = se^{sz} > c w  & \quad \mbox{for } z > 0,\label{positive_comp} \\
w = 1+e^{sz} \le 2
< \frac{1}{c}N 
 & \quad \mbox{for } z\le0.\label{negative_comp}
\end{align}
Integrating on each half strip (notation : $\int_{z>0}f:=\int_0^\infty\int_{[0,\la]} f (z,y,t)dy dz$)  and in time, we get 
\begin{align*}
\int _{z>0} w |\na \ps|^2 
 \le  & \int_{z>0} w |\na \ps_0|^2 + \int_0^t \int_{z>0} w \na (\dv \vp) \cdot \na \ps -
c\int_0^t \int _{z>0} w |\na \ps|^2 \\& -\frac s2\int_0^t  \int_{[0,\la]} w |\na \ps|^2(0, y) dy \quad (\mbox{by } \eqref{positive_comp}) \\%+\int_0^t \int_{z>0}{\ep\mbox{-terms}}\\
\le  &  \int _{z>0} w |\na \ps_0|^2 -\frac c2 \int_0^t \int_{z>0} w |\na \ps|^2 +
  C \int_0^t \int_{z>0} w |\na(\na \cdot \vp)|^2 \\& -\frac s2\int_0^t  \int_{[0,\la]} w |\na \ps|^2(0, y) dy
  %+\int_0^t \int_{z>0}{\ep\mbox{-terms}}
  \end{align*} and 
  \begin{align*}
 \int _{z<0} w |\na \ps|^2 \le & \int_{z<0} w |\na \ps_0|^2 + \int_0^t \int_{z<0} w \na (\dv \vp) \cdot \na \ps
\\& +\frac s2\int_0^t  \int_{[0,\la]} w |\na \ps|^2(0, y) dy\\% +\int_0^t \int_{z<0}{\ep\mbox{-terms}}\\
\le & \int_{z<0} w |\na \ps_0|^2 +  2\int_0^t \int_{z<0} |\na (\dv \vp)|| \na \ps |
 \\&  +\frac s2\int_0^t  \int_{[0,\la]} w |\na \ps|^2(0, y) dy \\% +\int_0^t \int_{z<0}{\ep\mbox{-terms}}\\
  \le &  \int_{z<0} w |\na \ps_0|^2 + C\int_0^t \int_{z<0} \fn{ |\na^2 \vp|^2}+ \int_0^t \int_{z<0} N |\na \ps|^2  
 \\&  +\frac s2\int_0^t  \int_{[0,\la]} w |\na \ps|^2(0, y) dy.  %+\int_0^t \int_{z<0}{\ep\mbox{-terms}} .
\end{align*}
Adding the above two then adding $\frac c2 \int_0^t \int_{z<0} w |\na \ps|^2 $ to the both sides, we get
\begin{align*}
\int w |\na \ps|^2 + \frac c2 \int_0^t \int w |\na \ps|^2 \le & \int w |\na \ps_0|^2 +  \int_0^t \int N|\na\ps|^2 \\
& + 
C\int_0^t \int w |\na^2 \vp|^2+\frac c2 \int_0^t \int_{z<0} w |\na \ps|^2 \\%  +\int_0^t \int {\ep\mbox{-terms}}\\
\le & \int w |\na \ps_0|^2 +  C\int_0^t \int N|\na\ps|^2  + 
C\int_0^t \int w |\na^2 \vp|^2 \quad (\mbox{by } \eqref{negative_comp}) \\
\le & C ( \| \na \ps_0\|_w^2 + \| \ps_0\|^2 + \| \vp_0\|_w^2)\\ & + C\sqrt{M(t)} \int_0^t \int \fn{ |\na \ps|^2}
+ C \int_0^t \int \fn{ |\na^2 \vp|^2} %+\int_0^t \int {\ep\mbox{-terms}}, 
\end{align*}
where we used the previous estimate \eqref{claim1_lemma1_} for the last inequality.   
\iffalse
%%%%%
%\ep begins
%%
For the $\ep ${-terms},  we estimate

\begin{align*}
&\int {\ep \mbox{-terms}}=\ep\int { w \na \ps\cdot\lr{\Del \na\ps -2  \na( P \cdot\na \ps) -   \na(|\na \ps|^2)}}\\
%&=\ep\int {w \na \ps\cdot \Del \na\ps + w \na \ps\cdot\lr{  -2  \na( P \cdot\na \ps) -   \na(|\na \ps|^2)}}\\
&=\ep\int \Big({ -w|\na^2\ps|^2 -w' \na \ps\cdot  \na\ps_z
- w \na \ps\cdot\lr{  2  \na( P \cdot\na \ps)} 
+  %w \na \ps\cdot
(w'\psi_z+w\Del\ps)
|\na \ps|^2}\Big)\\
&\leq-\ep\int {  w|\na^2\ps|^2 +C\ep\int \Big( w |\na \ps|| \na^2\ps| + w| \na \ps|( |\na^2\ps|+|\na\ps| )
+w |\na \ps|    |\na \ps|^2}+w|\na^2\ps||\na\ps|^2\Big)\\
&\leq-\frac \ep 4\int   w|\na^2\ps|^2 +C\ep\int w |\na \ps|^2 
 +C\ep\sqrt{M(t)}\int w |\na \ps|^2
\leq-\frac \ep 4\int   w|\na^2\ps|^2 +C\ep\int w |\na \ps|^2 
%  \\& \leq-\frac \ep 4\int   w|\na^2\ps|^2 +\frac{1}{2}\int w |\na \ps|^2 
\end{align*} where we used the estimate $|w'|\leq C|w|$ and for the last inequality, we assumed %$\ep$ and
 $M(t)$ small enough.

%%%%%
%\ep ends
%%%%
\fi
%Plugging this estimate,
Thus we have 
\begin{align*}
&\int w |\na \ps|^2 +  (\frac{c}{2}-C
%(\ep+
\sqrt{M(t)})
%) 
\int_0^t \int w |\na \ps|^2\\% +\frac{ \ep}{4} \int_0^t\int   w|\na^2\ps|^2 \\
&\le C ( \| \na \ps_0\|_w^2 + \| \ps_0\|^2 + \| \vp_0\|_w^2) %+ C\sqrt{M(t)} \int_0^t \int \fn{ |\na \ps|^2}
+ C \int_0^t \int \fn{ |\na^2 \vp|^2}.  
\end{align*} Then, by making
% $\ep$ and
  $M(t)$ small enough, it proves the estimate \eqref{claim2_lemma1_}.
%\begin{remark}
%This is the first place we need the smallness of $\ep$ essentially. %Indeed, we want to control $\ep\int { w \na \ps\cdot\lr{\Del \na\ps -2  \na( P \cdot\na \ps) }}$ but $C\ep\int w |\na \ps|^2$
%\end{remark}

\end{proof}
We are ready to obtain a closed energy estimate up to the first order derivatives:
\begin{lemma}\label{lemma01_} If $M(T)>0$ and $\la>0$
%, and $\ep\geq0$ 
are small enough, then
we have 
  \begin{equation}\begin{split}\label{eq_lemma01_}
\| \vp\|_{1,w}^2 + \| \ps\|^2    + \| \na \ps\|_w^2+  \int_0^t  \sum_{l = 1,2} \| \na^{l} \vp\|_w^2
+\int_0^t   \| \na \ps\|_w^2
%+\ep\int_0^t   \| \na^{2} \ps\|_w^2
%\\
%\quad 
 \le
 C  ( \|  \vp_0\|_{1,w}^2  +
 \| \na \ps_0\|_w^2+ \| \ps_0\|^2). 
\end{split}\end{equation}
\end{lemma}

\begin{proof}
Plugging the estimates \eqref{claim1_lemma1_} and \eqref{claim2_lemma1_} into \eqref{ineq:lem1_}, 
we  have
\begin{align*}
 &\| \na\vp\|_w^2 + \| \na\ps\|^2 + \int_0^t \| \na^2 \vp\|_w^2 \\ %+ \ep \int_0^t\|\na^2\ps\|^2 \\
 &\le  C ( \| \na \ps_0\|_w^2 + \| \ps_0\|^2 + \| \vp_0\|_{1,w}^2) + C\sqrt{{M(t)}}   \int_0^t \int \fn{|\na^2 \vp|^2}
\end{align*} which gives us 
\begin{align}\label{ineq:lem1_closed}
 &\| \na\vp\|_w^2 + \| \na\ps\|^2 + \int_0^t \| \na^2 \vp\|_w^2 % + \ep \int_0^t\|\na^2\ps\|^2
\le  C ( \| \na \ps_0\|_w^2 + \| \ps_0\|^2 + \| \vp_0\|_{1,w}^2) 
\end{align}
if we assume $M(t)$ small enough.
%Thus the first estimate of Lemma \ref{lemma1_} holds for a sufficiently small  $M(t)$.  
%In turns, , we have
In addition, from the estimate \eqref{claim2_lemma1_} with the above estimate \eqref{ineq:lem1_closed}, we get
\begin{align}\label{claims_lemma1_}
\int \fn{ |\na \ps|^2} + \int_0^t \int \fn{ |\na \ps|^2} %+  \ep\int_0^t \int \fn{ |\na^2 \ps|^2} 
\le C ( \|  \vp_0\|_{1,w}^2   + \| \na \ps_0\|_w^2
+ \| \ps_0\|^2).
\end{align}
  Adding Lemma \ref{lemma0_} and the above \eqref{ineq:lem1_closed} and \eqref{claims_lemma1_} and applying the estimate \eqref{claim2_lemma1_} into the sum, we have \eqref{eq_lemma01_}.
   Note that we have used $M(t)$ for $ \| \vp\|^2_{L^{\infty}} \le M(t)$, $ \| \dv\vp\|_{L^{\infty}}^2 \le M(t)$, and  $ \|\na \ps\|^2_{L^{\infty}} <M(t)$.

\end{proof}
 From now on, we will repeat the 
preceding energy estimates up to the third  order derivatives to have the weighted energy estimate \eqref{improve}. {Since the key ideas were already presented in detail on the previous lemmas, we do not split the higher order estimates into several lemmas as in the first order case (Lemma \ref{lemma0_}, \ref{lemma1_0}, \ref{lemma1_1}, \ref{lemma1_2} and \ref{lemma01_}) but state them   in the following single lemma, and we do not write its proof in detail.}

\begin{lemma}\label{lemma23_}If $M(T)>0$ and $\la>0$
%, and $\ep\ge 0$ 
are small enough, then 
for $k=2, 3$ we get
\begin{align*}
&\| \na^k \vp\|^2_w + \| \na^k \ps\|^2_w +   \int_0^t  { \|\na^{k+1}\vp\|_w^2}
+ \int_0^t  { \|\na^{k}\vp\|_w^2}
%+  \ep\int_0^t  { \|\na^{k+1}\ps\|_w^2}
 \le  C (  \| \vp_0\|_{k,w}^2+
 \| \na \ps_0\|_{k-1,w}^2+ \| \ps_0\|^2).
\end{align*}
\end{lemma}
\begin{remark}
{
\begin{enumerate}
\item The proof for $k=2$ is similar with that of the first order estimate except \eqref{la_0} and \eqref{la_1}. 
\item For $k=3$, we use 
$$\|(\na^2 \vp)/\sqrt{N}\|_{L^4}\leq \|\vp\|_{3,w}\leq \sqrt{M(t)}$$ (see \eqref{k3e1} and \eqref{k3e2})
as well as $\|\na \vp\|_{L^\infty}\leq \sqrt{M(t)}$.
\end{enumerate}
}
\end{remark}
 
%It remains to prove the above Lemma \ref{lemma23_}. 
 % The length of the following proof is quite long since we stated as a single lemma but the main ideas are identical with those of Lemma \ref{lemma0_}, \ref{lemma1_0}, \ref{lemma1_1}, \ref{lemma1_2} and \ref{lemma01_}.
\begin{proof}[Proof of Lemma \ref{lemma23_}]
Differentiating the $\vp,\ps$ equations $i+j$ times  in $y$ or $z$, we have 
\begin{align*}
&\ipa \jpa \vp_t - s \ipa \jpa \vp_z - \Del \ipa \jpa \vp \\
&\phantom{\ipa \jpa \ps_t - s \ipa\jpa \ps_z} =
  [ \ipa\jpa, N ] \na \ps + N \na  \ipa\jpa \ps  + \ipa\jpa( P \dv \vp) + \ipa \jpa( \dv \vp \na \ps),\\
&\ipa \jpa \ps_t - s \ipa\jpa \ps_z
% -\ep \Del \ipa\jpa\ps 
  = %-2\ep  \ipa\jpa(P \cdot\na \ps) -\ep  \ipa\jpa(|\na \ps|^2)+ 
   \dv \ipa\jpa \vp.
\end{align*}
Thus we get
\begin{align} \label{higher_}
& \frac 12 \ddt \int  \lr{ \fn{ | \ipa \jpa \vp|^2} + | \ipa \jpa \ps|^2 } 
+  \int \fn { | \na \ipa\jpa \vp |^2}
%+  \ep\int  { | \na \ipa\jpa \ps|^2}
\\
 &=
  %\underbrace
  {  \frac 12 \int \abs{\ipa \jpa \vp} \left( \fn{1} \right)''- \frac s2 \int \abs{\ipa \jpa \vp} \Fn'}
  %_{\leq  \frac{1}{2s}\int|\ipa \jpa\vp_{z}|^2(\frac{1}{N})'\leq \frac 12 \int\frac{|\ipa \jpa\vp_{z}|^2}{N}} 
 \nonumber\\
 &\underbrace{+\int  (    \ipa\jpa (N \na \ps) - N \ipa\jpa \na\ps  )\cdot  \fn { \ipa \jpa \vp} + \ipa \jpa (P \dv \vp) \cdot\fn{ \ipa\jpa \vp} }_{\mbox{ Quadratic term }}\nonumber\\
& \underbrace{+ \ipa\jpa ( \dv \vp \na \ps) \fn{ \ipa \jpa \vp}}_{\mbox{ Cubic term }}. \nonumber\\
%& \underbrace{ -%\underbrace{ 2\ep \int \ipa\jpa(P\cdot \na\ps) \ipa\jpa \psi }_
%%{= \,+ 
%2\ep \int \ipa\jpa(\calP \ps_z) \ipa\jpa\psi
%%}
%- \ep \int \ipa\jpa(|\na \ps|^2) \ipa\jpa\ps}_{\ep-\mbox{term}}. \nonumber
\end{align}

%%%%% 
%wait
%%%
$\bullet$ Case $ k= i+j =2$\\

First, we estimate
\begin{align*}
&\frac 12 \int |\na^2\vp|^2 \lr{ \fn{1}}''-\frac s 2 \int |\na^2\varphi|^2 \left( \fn{1}\right)'
\le -\int \na^2\vp\cdot\na^2\vp_z  \left( \fn{1}\right)'+\frac s 2 \int |\na^2\varphi|^2 \left|\left( \fn{1}\right)'\right|\\
&
\le  C\int |\na^2\vp||\na^2\vp_z |  \left( \fn{1}\right) +\frac {Cs}{2} \int |\na^2\varphi|^2  \left( \fn{1}\right) 
\le  \frac{1}{4}\int  |\na^3\vp |^2\left( \fn{1}\right)+C \int |\na^2\varphi|^2  \left( \fn{1}\right).
%&\le  \frac{1}{4}\int  |\vp_z |^2\left( \fn{1}\right)+C\la^2(C+\frac {Cs}{2} ) \int |\varphi_y|^2  \left( \fn{1}\right)\\
 \end{align*}

%Quadratic terms are 
%\[ \int  \fn{N''} \vp_{zz} \na \ps,   \quad  \int \fn{N'} \vp_{zz} \na \ps_z, \quad \int \fn{N'} \vp_{yz}\na\ps_y,\]
%\[ \int \fn{P''} \dv \vp \vp_{zz}, \quad  \fn{P'} \dv \vp ( \vp_{zz} + \vp_{yz}), \quad  \fn{P} (  \vp_{zz} \dv\vp_{zz}
% +  \vp_{yz} \dv \vp_{yz}+ \vp_{yy} \dv\vp_{yy}).\]

%Also from $ \fn{N'} < C$, it holds that
%\[  \pa_l  \lr{\fn{ f}}  \le     \fn{ \pa_l f}  +  \left|  \fn{f}  \right |.\]
%In particular it holds that 
%%\begin{align}\label{inequality:ng}
%%\| \fn{ fg }\|_{L^2} &\le C \| \frac{  f}{\sqrt N}\|_{L^4}\| \frac{ g}{\sqrt N}\|_{L^4}\\
%%& \le C \| \frac{ f}{\sqrt N}\|^{\frac 12}_{L^2} \| \na \lr{\frac{ f}{\sqrt N}}\|^{\frac 12}_{L^2}
%%\| \frac{ g}{\sqrt N}\|^{\frac 12}_{L^2} \| \na \lr{\frac{ g}{\sqrt N}}\|^{\frac 12}_{L^2}\\
%%& \le C \int  \fn{|\na f|^2} ( \int \fn{ |\na f|^2} + \fn{f^2}) 
%%\end{align}
%\begin{align}\label{inequality:ng_}
% \| \frac{  f}{\sqrt N}\|^2_{L^4} & \le C \| \frac{ f}{\sqrt N}\|_{L^2} \| \na \lr{\frac{ f}{\sqrt N}}\|_{L^2}
%  \le C  \| \frac{ f}{\sqrt N}\|_{L^2} \lr{ \| \frac{\na f}{\sqrt N} \|_{L^2} + \| \frac {f}{\sqrt N}\|_{L^2}}.
%\end{align}
What it follows we do not distinguish $\pa_y$ and $ \pa_z$ derivatives.\\
The quadratic terms are symbolically
\begin{align*}
 \fn{N''} \na \ps \na^2 \vp, \quad  \fn{N'} \na^2 \ps \na^2 \vp, %\quad \\
 \fn{P''} \na \vp \na^2 \vp, \quad   \fn{P'} \na^2 \vp \na^2 \vp  \quad \mbox{ and } \quad \fn{P} \na^3 \vp \na^2 \vp.\\
 %\na^2 \vp \na^2 \ps, \quad  \fn{1} \na \vp \na^2 \vp, \quad  \fn{1} \na^2 \vp \na^3 \vp, 
\end{align*}
By Lemma \ref{wave_prop}, 
%\[ |{N^{ (k)}}| < C, \quad   |{P^{(k)}}| <C,\quad \mbox{ for } k\le 2,\quad 
% \Big|\fn{N'}\Big|=|P+s|\leq C, \mbox{ and }\Big|(\fn{N'})'\Big|=|P'|\leq C.\]
%($N^{(k)}$ is the $k$th derivative of $N$).
%So
 these   terms are estimated by 
\begin{align*}
& C\int \Big|\fn{N''}\Big|| \na \ps|| \na^2 \vp| +
 \Big|\fn{P''} \Big||\na \vp ||\na^2 \vp|+ \Big|\fn{P'} \Big||\na^2 \vp ||\na^2 \vp|\\
% &\leq  C \|  \frac{\na^2\vp }{ \sqrt{N}} \|\lr{\| \frac{\na\ps }{\sqrt{N}}\| 
% +\| \frac{\na\vp }{\sqrt{N}}\|+\| \frac{\na^2\vp }{\sqrt{N}}\|}\\
 &\leq  C \lr{\| \frac{\na\ps }{\sqrt{N}}\|^2
 +\| \frac{\na\vp }{\sqrt{N}}\|^2+\| \frac{\na^2\vp }{\sqrt{N}}\|^2},
 %\na^2 \vp \na^2 \ps, \quad  \fn{1} \na \vp \na^2 \vp, \quad  \fn{1} \na^2 \vp \na^3 \vp, 
\end{align*} 
\begin{align*}
C\int \Big|\frac{P}{N}\Big|  |\na^3 \vp| |\na^2\vp|   \le
 C \|  \frac{\na^3\vp }{ \sqrt{N}} \|\| \frac{\na^2\vp }{\sqrt{N}}\|  \le
 \frac{1}{8} \|  \frac{\na^3\vp }{ \sqrt{N}} \|^2 +C\| \frac{\na^2\vp }{\sqrt{N}}\|^2
\end{align*} and 
\begin{align*}
& C\Big|\int \fn{N'} \na^2 \ps \na^2 \vp\Big|\leq C \int  | \na \ps| |\na^2 \vp|
+C\int   |\na \ps ||\na^3 \vp|\\
 &\leq  C \|  {\na\ps } \|\lr{\|  {\na^2\vp } \|+\|  {\na^3\vp } \|}\leq 
 C \|  {\na\ps } \|^2+ C\|  {\na^2\vp } \|^2+\frac 1 8 \|  {\na^3\vp } \|^2
\end{align*} where we used integration by parts for the last estimate.

%So by Lemma \ref{lemma01_}, we have
%\begin{align*} &\int \fn{ |\na^2 \vp|^2} +\int |\na^2 \ps|^2 + \int_0^t \int \fn{ |\na^3 \vp|^2} +
% \ep\int_0^t \int { |\na^3 \ps|^2}\\
% &\le C( \|  \vp_0\|_{2,w}^2  +
% \| \na \ps_0\|_{1,w}^2+ \| \ps_0\|^2) \\% +C\int_0^t \int N |\na^2 \ps|^2  \\
% & +\mbox{cubic terms } +\ep\mbox{ terms }
% \end{align*}

The cubic terms 
%$\ipa\jpa ( \dv \vp \na \ps) \fn{ \ipa \jpa \vp}$
 are symbolically written as 
$ \na^2 (\na \vp \na \ps) \fn{\na^2\vp}$. By using integration by parts once, it can be written as 
\begin{align*}
& \na  (\na \vp \na \ps) \fn{\na^3\vp} \quad \mbox{and } \quad \na  (\na \vp \na \ps) {\na^2\vp}%\underbrace
{(\frac{1}{N})'} .
\end{align*} So by assuming smallness of $M(t)$, these terms are estimated by
\begin{align*}
&   C\int (|\na^2 \vp|| \na \ps|+|\na \vp ||\na^2 \ps|) \fn{|\na^3\vp|} \leq
C\sqrt{M(t)}\int (|\na^2 \vp| + |\na^2 \ps|) \fn{|\na^3\vp|}\\
&\leq  C   \|\frac{\na^2 \vp}{\sqrt{N}}\|^2+ C \sqrt{M(t)}\|\frac{\na^2 \ps}{\sqrt{N}}\|^2  +\frac{1}{8}\|\frac{\na^3\vp}{\sqrt{N}}\|^2
\end{align*} and

\begin{align*}
&   C\int(|\na^2 \vp ||\na \ps|
+|\na \vp|| \na^2 \ps|){\fn{|\na^2\vp|}}\leq C  \|\frac{\na^2 \vp}{\sqrt{N}}\|^2+ C \sqrt{M(t)}\|\frac{\na^2 \ps}{\sqrt{N}}\|^2.
\end{align*} Note that we  used $M(t)$ for   $ \| \na\vp\|^2_{L^{\infty}} \le M(t)$ and  $ \|\na \ps\|^2_{L^{\infty}} <M(t)$.

%So by Lemma \ref{lemma01_}, we have
%\begin{align*} &\int \fn{ |\na^2 \vp|^2} +\int |\na^2 \ps|^2 + \int_0^t \int \fn{ |\na^3 \vp|^2} +
% \ep\int_0^t \int{ |\na^3 \ps|^2}\\
% &\le C( \|  \vp_0\|_{2,w}^2  +
% \| \na \ps_0\|_{1,w}^2+ \| \ps_0\|^2) % +C\int_0^t \int N |\na^2 \ps|^2 
% +CM(t)\int_0^t \int  \fn{|\na^2 \ps|^2} \\
% & +\ep\mbox{ terms }
% \end{align*}
 
 %%%
 %ep begins
 %%%
 \iffalse
 For the $\ep$ terms, %$-\underbrace{ 2\ep \int \ipa\jpa(P\cdot \na\ps) \ipa\jpa \psi }_
%{= \,+ 2\ep \int \ipa\jpa(\calP \ps_z) \ipa\jpa\psi}
%- \ep \int \ipa\jpa(|\na \ps|^2) \ipa\jpa\ps$,
 we can write them  symbolically:
$$ 
    \ep \int \na^2(\calP \ps_z) \na^2\psi  \quad \mbox{and } \quad
 \ep \int\nabla^2(|\na \ps|^2) \nabla^2\ps.
$$
After intergration by parts, we can estimate them by
\begin{align*}
 & C\ep \int |\na(\calP \ps_z)| |\na^3\psi|\leq
 C\ep \int (|\na \ps| +|\na^2 \ps|) |\na^3\psi|
 \leq  C\ep ( \|\nabla\ps\|^2+ \|\nabla^2\ps\|^2) +\frac{\ep}{4}\|\nabla^3\ps\|^2\\
\end{align*}
 and
\begin{align*}&C\ep \int|\nabla(|\na \ps|^2)| |\nabla^3\ps|\leq C\ep\int|\nabla\ps||\nabla^2\ps| |\nabla^3\ps|\\
&\leq C\ep  \sqrt{M(t)}\int |\nabla^2\ps| |\nabla^3\ps|\leq  C\ep M(t) \|\nabla^2\ps\|^2 +\frac{\ep}{4}\|\nabla^3\ps\|^2.
\end{align*}
 %%%
 %ep ends
 %%%
 \fi
 Up to now, by Lemma \ref{lemma01_}, we estimate \eqref{higher_} by
\begin{align}\label{beforeclaims_lemma23_} &\int \fn{ |\na^2 \vp|^2} +\int |\na^2 \ps|^2 + \int_0^t \int \fn{ |\na^3 \vp|^2} % + \ep\int_0^t \int{ |\na^3 \ps|^2}
\\
 &\le C( \|  \vp_0\|_{2,w}^2  +
 \| \na \ps_0\|_{1,w}^2+ \| \ps_0\|^2) 
 %+C\int_0^t \int N |\na^2 \ps|^2 
 +C \sqrt{M(t)}\int_0^t \int  \fn{|\na^2 \ps|^2}.  \nonumber 
 \end{align} This estimate is the second order version of  Lemma \ref{lemma1_0}.\\
 
 Now we claim the following two estimates which are the second order versions of Lemma \ref{lemma1_1} and
 \ref{lemma1_2}:
 
 \begin{align} 
 \label{claim1_lemma23_}
& \int_0^t \int N |\na^2 \ps|^2 %+ \ep\int _0^ t \int   {|\na \Del \ps|^2 } 
% \\&\qquad\qquad\qquad \qquad 
\le C \lr{   \|  \vp_0\|_{1,w}^2  +
 \| \na \ps_0\|_w^2+ \| \ps_0\|^2+\|\Del\ps_0\|^2} + C\sqrt{M(t)} \int_0^t  \int \fn{ |\na^2 \ps |^2}, \\
% +\la \int_0^t \int \fn{ |\na^3 \vp|^2}},  \\
& \int \fn {|\na^2  \ps|^2} +   \int _0^ t \int \fn {|\na^2 \ps|^2 } 
%+ \ep\int _0^ t \int \fn {|\na^3 \ps|^2 } 
\le  C    ( \| \na \ps_0 \|_{1,w}^2 + \|  \ps _0\|^2 + \|  \vp_0\|_{1,w}^2) + C\int_0^t \int \fn{ |\na^3\vp|^2}. \label{claim2_lemma23_}
\end{align}
 As we did in Lemma \ref{lemma1_1} and
 \ref{lemma1_2}, our plan is to prove \eqref{claim1_lemma23_} first and to use the result in order to get \eqref{claim2_lemma23_}. Then we will close the  estimate \eqref{beforeclaims_lemma23_} by using them. \\
 
 $\bullet$ Proof of \eqref{claim1_lemma23_}\\
 
Taking $ \dv$ to $\vp$ equation, we have
\begin{align*}
&\dv \vp_t - s \dv \vp_z - \Del \dv \vp \\
& = \quad N \Del \ps + \underbrace{ \na N \cdot \na\ps  +  \calP' \dv \vp + \calP ( \dv \vp)_z  
 +  \dv (\dv \vp \na\ps)}_{R_1 }.
\end{align*}
We multiply  $\Del \ps$ on the both sides and plug it in the equation
\[ \Del  \ps_t - s\Del \ps_z  %-\ep \Del\Del \ps 
=  %\Del\lr{-2\ep  P \cdot\na \ps -\ep  |\na \ps|^2 }+ 
\Del  (\dv \vp)\]
in order to get 
\begin{align*}
N |\Del \ps|^2 &=  ( \dv \vp_t - s\dv \vp_z - \Del \dv \vp) \Del \ps - \mbox{ R}_1\Del\ps\\
&= (\dv \vp \Del \ps )_t - \dv \vp \Del \ps_t - s \dv \vp_z \Del \ps -\underbrace{ \Del \dv \vp \Del \ps}_{(C)} -\mbox{ R}_1\Del\ps\\
&= (\dv \vp \Del \ps )_t - \dv \vp( s \Del \ps_z + \Del \dv \vp
) - s \dv \vp_z \Del \ps -(C)
%\\&\quad
-\mbox{ R}_1\Del\ps.
%-\ep\dv \vp(  \Del\underbrace{\lr{\Del \ps -2   P \cdot\na \ps -   |\na \ps|^2 }}_{R_2}).
\end{align*}

 % Recall the $\ps$ equation,
%\[  \ps_t - s \ps_z -\ep \Del \ps =  -2\ep  P \cdot\na \ps -\ep  |\na \ps|^2 + \dv\varphi
%.\]

 % Plugging the $\ps$ equation in $\ps$ equation,
 For $(C)$, we get
\begin{align*} \int(C)=&\int  \Del (\dv \vp) \Del \ps =\int  \Del (\ps_t - s \ps_z ) \Del \ps  %- \int  \Del (\ep \Del \ps-2\ep  P \cdot\na \ps -\ep  |\na \ps|^2) \Del \ps
\\
 & =\frac 12 \ddt \int |\Del \ps|^2. 
 %- \ep\int  \Del ( \Del \ps-2   P \cdot\na \ps -   |\na \ps|^2) \Del \ps
% =\frac 12 \ddt \int |\Del \ps|^2  %- \ep\int  \Del R_2 \Del \ps.
 \end{align*}
 So, integrating on the strip, we have
\begin{align*}   \int N |\Del \ps|^2 &= \ddt \int \dv \vp \Del \ps + \int |\na  \dv \vp|^2 - \frac 12 \ddt \int |\Del \ps|^2     - \int \mbox{ R}_1\Del\ps. %+ \ep\int  \Del (  \mbox{R}_2 ) \Del \ps -\ep \int \dv \vp(  \Del\mbox{R}_2).
 \end{align*}
 
 Note that 
 \begin{align*}
 \int_0^t\lr{\ddt \int \dv\vp \Del \ps - \frac 12 \ddt \int  |\Del\ps|^2}\leq {C(\|\na\vp(t)\|^2 
 +\|\Del \ps_0\|^2+\|\na\vp_0\|^2)}&
 \end{align*} by 
$ \int  |\dv\vp \Del \ps| \le C\| \na\vp \| ^2 + \frac 12   \| \Del \ps\|^2$.
%and $\int |\na  \dv \vp|^2\leq C\int |\na^2   \vp|^2/N.$
 
 %%%%%%
 %after this point...
 %%%%
 The terms in $R_1$ are estimated as follows;
 \begin{align*}
& \int   ( \na N \cdot \na\ps) \Del \ps \le C\int \fn{ |\na\ps|^2 } + \frac 14 \int  N |\Del \ps|^2, \\
& \int \calP' \dv \vp \Del \ps + \calP ( \dv \vp_z) \Del \ps \le 
C \lr{ \int \fn{ |\na \vp|^2} + \int \fn{|\na^2 \vp|^2}} + \frac 18 \int N |\Del \ps|^2,\\
&\int \dv (\dv\vp \na \ps ) \Del \ps  \le C ( \| \na \ps\|_{L^{\infty}} + \| \na \vp \|_{L^{\infty}})
 \lr{ C \int \fn{ |\na^2 \vp|^2} + C\int \fn{ |\Del \ps|^2} + \frac 14 \int N |\Del \ps|^2}\\
&\quad \qquad\qquad\qquad\qquad\leq C  
  \int \fn{ |\na^2 \vp|^2} + C\sqrt{M(t)}\int \fn{ |\Del \ps|^2} + \frac 14 \int N |\Del \ps|^2
 \end{align*} by assuming smallness of $M(t)$.

Collecting the above estimates and using  Lemma \ref{lemma01_}, we have
 \begin{align*}&\int_0^t \int N |\Del \ps|^2
%  +\ep\int_0^t \int  |\na\Del\ps|^2 
  \le    C  ( \|  \vp_0\|_{1,w}^2  +
 \| \na \ps_0\|_w^2+ \| \ps_0\|^2+\|\Del\ps_0\|^2) 
 + C\sqrt{M(t)} \int_0^t \int \fn{|\Del \ps|^2}.\end{align*}

To get \eqref{claim1_lemma23_} from the above estimate, we have to control 
$ \int N |\na^2\ps|^2$ by $\int N |\Del \ps|^2$ with lower order estimates. 
Observe 
\begin{align*}&\int N |\Del \ps|^2 = \int N( (\pa_{zz}\ps)^2 +  (\pa_{yy}\ps)^2+  2\pa_{zz}\ps \pa_{yy}\ps)\\
&= \underbrace{\int N( (\pa_{zz}\ps)^2 +  (\pa_{yy}\ps)^2+  2(\pa_{zy}\ps)^2)}_{\int N |\na^2\ps|^2}
-2 \int N' \pa_z\ps \pa_{yy} \ps
\end{align*} 
 
and  
\begin{equation} \label{la_0}\Big| \int N' \pa_1 \ps \pa_{22} \ps \Big|= \Big|-\int(P+s) N \pa_1\ps \pa_{22}\ps\Big|
\le \frac 14 \int N( \pa_{22}\ps)^2
+ C\int N |\na\ps|^2  \end{equation} by $ N'=-(P+s)N$
from  \eqref{np_copy}.
Thus  we get
\begin{align} \label{la_1} \int N |\na^2\ps|^2
&\leq 2 \int N( (\pa_{11}\ps)^2 + \frac 12 (\pa_{22}\ps)^2+  2(\pa_{12}\ps)^2)\\
&= 2\int N |\Del \ps|^2+4\int N' \pa_1 \ps \pa_{22} \ps -  \int N( \pa_{22}\ps)^2 \nonumber\\
&\leq 2\int N |\Del \ps|^2  +  \int N( \pa_{22}\ps)^2
+ C\int N |\na\ps|^2 -  \int N( \pa_{22}\ps)^2\nonumber \\
&\leq 2\int N |\Del \ps|^2   
+ C\int N |\na\ps|^2.\nonumber   
  \end{align} 
 So we have
 \[ \int_0^t\int N |\na^2\ps|^2  \le 
   C ( \|  \vp_0\|_{1,w}^2  +
 \| \na \ps_0\|_w^2+ \| \ps_0\|^2)+ C\int_0^t\int N |\Del \ps|^2 \] 
by Lemma \ref{lemma01_}. Thus we proved \eqref{claim1_lemma23_}.\\

$\bullet$ Proof of \eqref{claim2_lemma23_}\\

Multiplying $ w \na^2 \ps$ to the equation $$ \na^2 \ps_t - s \na^2\ps_z
% -\ep \Del \na^2\ps 
 = \na^2 (\dv \vp),
 %-2\ep \na^2( P \cdot\na \ps) -\ep  \na^2(|\na \ps|^2), 
$$ we have
\begin{align*} 
&\frac 12 (w |\na^2 \ps|^2)_t - \frac s2 ( w |\na^2 \ps|^2)_z + \frac s2 w' |\na^2 \ps|^2 
= w \na^2 (\dv \vp) \cdot \na^2 \ps. 
%+\underbrace{\ep w \na^2 \ps\cdot\lr{\Del \na^2\ps -2  \na^2( P \cdot\na \ps) -   \na^2(|\na \ps|^2)}}_{\ep- \mbox{terms}}.
\end{align*}
 
Recall that for some $c>0$, we have
\begin{align*}
w' = se^{sz} > c w  & \quad \mbox{for } z > 0, \\
w = 1+e^{sz} \le 2
< \frac{1}{c}N 
 & \quad \mbox{for } z\le0.
\end{align*}
 
Integrating on each half strip % (notation : $\int_{z>0}f:=\int_0^\infty\int_{[0,\la]} f (z,y,t)dy dz$) 
and in time,  we get
\begin{align*}
\int _{z>0} w |\na^2 \ps|^2 &
 \le  \int_{z>0} w |\na^2 \ps_0|^2 + \int_0^t \int_{z>0} w \na^2 (\dv \vp) \cdot \na^2 \ps -
c\int_0^t \int _{z>0} w |\na^2 \ps|^2 \\&\quad -\frac s2\int_0^t  \int_{[0,\la]} w |\na^2 \ps|^2(0, y) dy %+\int_0^t \int_{z>0}{\ep\mbox{-terms}}
\\
& \le  \int _{z>0} w |\na^2 \ps_0|^2 -\frac c2 \int_0^t \int_{z>0} w |\na^2 \ps|^2 +
  C \int_0^t \int_{z>0} w |\na^2(\na \cdot \vp)|^2 \\&\quad -\frac s2\int_0^t  \int_{[0,\la]} w |\na^2 \ps|^2(0, y) dy,
 % +\int_0^t \int_{z>0}{\ep\mbox{-terms}},
  \end{align*}
  \begin{align*}
 \int _{z<0} w |\na^2 \ps|^2 \le & \int_{z<0} w |\na^2 \ps_0|^2 + \int_0^t \int_{z<0} w \na^2 (\dv \vp) \cdot \na^2 \ps \\&\quad
 +\frac s2\int_0^t  \int_{[0,\la]} w |\na^2 \ps|^2(0, y) dy
 %+\int_0^t \int_{z<0}{\ep\mbox{-terms}}
 \\
&\le \int_{z<0} w |\na^2 \ps_0|^2 + C\int_0^t \int_{z<0}  | \na^3 \vp|||\na^2 \ps|  \\&\quad
  +\frac s2\int_0^t  \int_{[0,\la]} w |\na^2 \ps|^2(0, y) dy
  % +\int_0^t \int_{z<0}{\ep\mbox{-terms}}
  \\
  &\le  \int_{z<0} w |\na^2 \ps_0|^2 + C\int_0^t \int_{z<0}  N|\na^2 \ps|^2  
  + C\int_0^t \int_{z<0}  { \fn{|\na^3 \vp|^2}} \\&\quad
  +\frac s2\int_0^t  \int_{[0,\la]} w |\na^2 \ps|^2(0, y) dy.
  % +\int_0^t \int_{z<0}{\ep\mbox{-terms}}.
\end{align*}
Adding those, we get
\begin{align*}
&\int w |\na^2 \ps|^2 + \frac c2 \int_0^t \int w |\na^2 \ps|^2 \le \int w |\na^2 \ps_0|^2
+ C\int_0^t \int  \Big(N  |\na^2\ps|^2 + 
  w |\na^3 \vp|^2 % +  {\ep\mbox{-terms}}
  \Big).
%&\le C  \| \na \ps_0\|_{1,w}^2   +  C \int_0^t \int  {N |\na^2 \ps|^2}+ C \int_0^t \int \fn{ |\na^3 \vp|^2}+\int_0^t \int {\ep\mbox{-terms}}. 
\end{align*}
  
\iffalse
%%%%%
%\ep begins
%%
For the $\ep${-terms}, as before, we estimate

\begin{align*}
&\int {\ep\mbox{-terms}}=\ep\int { w \na^2 \ps\cdot\lr{\Del \na^2\ps -2  \na^2( P \cdot\na \ps) -   \na^2(|\na \ps|^2)}}\\
%&=\ep\int {w \na \ps\cdot \Del \na\ps + w \na \ps\cdot\lr{  -2  \na( P \cdot\na \ps) -   \na(|\na \ps|^2)}}\\
&=\ep\int { \Big(-w|\na^3\ps|^2 -w' \na^2 \ps\cdot  \na^2\ps_z
- w \na^2 \ps\cdot\lr{  2  \na^2( P \cdot\na \ps)} 
+  %w \na \ps\cdot
(w'\na\psi_z+w\Del\na\ps)
\na(|\na \ps|^2)}\Big)\\
&\leq-\ep\int   w|\na^3\ps|^2 \\&\,\, +C\ep\int \Big(w |\na^2 \ps|| \na^3\ps| + w| \na^2 \ps|( |\na^3\ps|+|\na^2\ps|+|\na\ps| ) 
+w|\na \ps|( |\na^2 \ps|    |\na^2 \ps| +|\na^3\ps| |\na^2 \ps|)\Big)\\
&\leq-\frac \ep 4\int   w|\na^3\ps|^2 +C\ep\int w |\na^2 \ps|^2 +C\ep\int w |\na \ps|^2 
 %\leq-\frac \ep 4\int   w|\na^3\ps|^2 +\frac{1}{2}\int w (|\na^2 \ps|^2 +|\na \ps|^2 )
\end{align*} where we used the estimate $|w'|\leq C|w|$ and for the last inequality, we assumed 
%$\ep$ and 
$M(t)$ small enough.\\
\fi
%%%%%
%\ep ends
%%%%
Collecting the above estimates, and using Lemma \ref{lemma01_} and the previous claim \eqref{claim1_lemma23_}, we have 
\begin{align*}
&\int w |\na^2 \ps|^2 +   (\frac{c}{2}-C
%(\ep+
\sqrt{M(t)})
%)
 \int_0^t \int w |\na^2 \ps|^2 
 %+ \frac\ep4 \int_0^t\int   w|\na^3\ps|^2 
 \\
&\le C ( \| \na \ps_0\|_{1,w}^2 + \| \ps_0\|^2 + \| \vp_0\|_{1,w}^2) 
%+ C \int_0^t \int { N|\na^2 \ps|^2}
+ C \int_0^t \int \fn{ |\na^3 \vp|^2} . 
\end{align*}

  Then, by making 
  %$\ep$ and
   $M(t)$ small enough, it proves the claim \eqref{claim2_lemma23_}.\\

Now we are ready to finish this proof for Lemma \ref{lemma23_} for the second order.
Plugging \eqref{claim2_lemma23_}  into \eqref{beforeclaims_lemma23_} with $M(t)$ small, we  have
\begin{align*}
 &\| \na^2\vp\|_w^2 + \| \na^2\ps\|^2 + \int_0^t \| \na^3 \vp\|_w^2 
 % + \ep \int_0^t\|\na^3\ps\|^2 
% \\
% &
 \le  C ( \| \na \ps_0\|_{1,w}^2 + \| \ps_0\|^2 + \| \vp_0\|_{2,w}^2). %+ C{M(t)}   \int_0^t \int \fn{|\na^2 \vp|^2}
\end{align*}
%Thus the first estimate of Lemma \ref{lemma1_} holds for a sufficiently small  $M(t)$.  
In turns,  we have
\begin{align}%\label{claims_lemma1_}
\int \fn{ |\na^2 \ps|^2} + \int_0^t \int \fn{ |\na^2 \ps|^2}
% + \ep\int_0^t \int \fn{ |\na^3 \ps|^2} 
 \le C ( \|  \vp_0\|_{2,w}^2   + \| \na \ps_0\|_{1,w}^2
+ \| \ps_0\|^2).
\end{align}   This proves Lemma \ref{lemma23_} for case $k=i+j=2$.

 \begin{remark}

Together with \ref{lemma01_}, we have proved 
\begin{align}\label{lemma2_} &
\|  \vp\|_{2,w}^2  +
 \| \na \ps\|_{1,w}^2+ \| \ps\|^2
%\int \fn{ |\na^2\vp|^2} +\int \fn {|\na^2  \ps|^2} 
+  \int_0^t  \sum_{l = 1,2,3} \| \na^{l} \vp\|_w^2
+\int_0^t   \sum_{l = 1,2}\| \na^l \ps\|_w^2
%+\ep\int_0^t  \sum_{l = 1,2,3} \| \na^{l} \ps\|_w^2
\\
 &\le C( \|  \vp_0\|_{2,w}^2  +
 \| \na \ps_0\|_{1,w}^2+ \| \ps_0\|^2). \nonumber
\end{align}
  \end{remark}
 \ \\

$\bullet$ Case $ k=i + j =3$\\
\ \\
We repeat the previous argument for $k=3$. 

{When we see either $N'''$ or $P'''$, we use the integration by parts to reduce the order of $N$ or $P$ by 1 (Recall Lemma \ref{wave_prop}).} Then by a similar argument, 
%So by \eqref{lemma2_}, 
we have
\begin{align*} &\int \fn{ |\na^3 \vp|^2} +\int |\na^3 \ps|^2 + \int_0^t \int \fn{ |\na^4 \vp|^2}  
% +\ep\int_0^t \int { |\na^4 \ps|^2}
 \\
 &\le C( \|  \vp_0\|_{3,w}^2  + \| \na^3 \ps_0\|^2+
 \| \na \ps_0\|_{1,w}^2+ \| \ps_0\|^2)  % +C\int_0^t \int N |\na^2 \ps|^2  \\
  +\mbox{the cubic terms }.% +\ep\mbox{-terms }.
 \end{align*}
All cubic terms can be estimated very similarly except
$ \int \na^2\vp \na^2 \ps \fn{\na^4 \vp}$.
 Note that by the Sobolev embedding, 
$$\|f\|_{L^4}\leq C(\|f\|_{L^2}+\|\na f\|_{L^2}).$$
So we  estimate
\begin{align}\label{k3e1}
&C\Big|\int \na^2\vp \na^2 \ps \fn{\na^4 \vp}\Big|  
%\le C \int \Big|  {\na^2 \ps  }   \frac{\na^2 \vp  }{   \sqrt {N}}\frac{\na^4 \vp}{\sqrt N} \Big|
\le C \|    {\na^2 \ps} \|_{L^4}
\|  \frac {\na^2 \vp}{\sqrt  N}\|_{L^4}\| \|  \frac {\na^4 \vp}{\sqrt N}\|_{L^2} \\
& \le \frac 18 \|  \frac {\na^4 \vp}{\sqrt N}\|^2_{L^2} + 
C    \lr{ \|  \na\left( \frac {\na^2 \vp}{\sqrt  N}\right)\|_{L^2} +\|  \frac {\na^2 \vp}{\sqrt  N}\|_{L^2}}^2 
\cdot \lr{ \|      {\na^3 \ps}  \|_{L^2} +\|    {\na^2 \ps} \|_{L^2}}^2 \nonumber\\
& \le \frac 18 \|  \frac {\na^4 \vp}{\sqrt N}\|^2_{L^2} + 
C    \lr{ \|    \frac {|\na^3 \vp|}{\sqrt N}+|\na^2\vp|
{|(\frac{1}{\sqrt{N}})'|}\|_{L^2} +C\|  \frac {\na^2 \vp}{\sqrt N}\|_{L^2}}^2 
\cdot \lr{ \|      {\na^3 \ps}  \|_{L^2} +{\|    {\na^2 \ps} \|_{L^2}}}^2. \nonumber
\end{align}
Using ${|(\frac{1}{\sqrt{N}})'|}{\leq \frac{C}{\sqrt{N}}}$ (Lemma \ref{wave_prop}), we get
\begin{align}\label{k3e2}& \lr{ \|    \frac {|\na^3 \vp|}{\sqrt N}+|\na^2\vp|
 {|(\frac{1}{\sqrt{N}})'|} \|_{L^2}
+C\|  \frac {\na^2 \vp}{\sqrt N}\|_{L^2}}^2 
\leq 
C \lr{\|    \frac {\na^3 \vp}{\sqrt N}  \|^2_{L^2} +\|  \frac {\na^2 \vp}{\sqrt N}\|_{L^2}^2}
\leq CM(t).
\end{align}
%So by \eqref{lemma2_}, we have
%\begin{align*} &\int \fn{ |\na^3 \vp|^2} +\int |\na^3 \ps|^2 + \int_0^t \int \fn{ |\na^4 \vp|^2} +
% \ep\int_0^t \int { |\na^4 \ps|^2}\\
% &\le C( \|  \vp_0\|_{3,w}^2  + \| \na^3 \ps_0\|^2+
% \| \na \ps_0\|_{1,w}^2+ \| \ps_0\|^2)+ C\sqrt{M(t)}\int_0^t \int  \fn{ |\na^3\ps|^2} \\% +C\int_0^t \int N |\na^2 \ps|^2  \\
% & +\ep\mbox{-terms }.
% \end{align*}

  \iffalse
 For the $\ep$-terms, we can write them  symbolically:
$$ 
    \ep \int \na^3(\calP \ps_z) \na^3\psi  \quad\mbox{ and }\quad
 \ep \int\nabla^3(|\na \ps|^2) \nabla^3\ps.
$$
After integration by parts, we can estimate these terms by
\begin{align*}
 & C\ep \int |\na^2(\calP \ps_z)| |\na^4\psi|\leq
 C\ep \int (|\na \ps| +|\na^2 \ps|+|\na^3 \ps|) |\na^4\psi|\\
 &\leq  C\ep ( \|\nabla\ps\|^2+ \|\nabla^2\ps\|^2+\|\nabla^3\ps\|^2) +\frac{\ep}{4}\|\nabla^4\ps\|^2\\
\end{align*}
 and
\begin{align*}&C\ep \int|\na^2(|\na \ps|^2)| |\na^4\ps|
\leq C\ep\int(|\nabla\ps||\nabla^3\ps|+|\nabla^2\ps||\nabla^2\ps|) |\nabla^4\ps|\\
&\leq C\ep \sqrt{M(t)}\int \Big(|\nabla^3\ps| |\nabla^4\ps|+\|\na^2\ps\|^2_{L^4}\cdot\|\na^4\ps\|_{L^2}\Big)\\
&\leq  C\ep \sqrt{M(t)} \|\nabla^3\ps\|^2 +\frac{\ep}{8}\|\nabla^4\ps\|^2+C\ep\underbrace{(\|\na^3\ps\|^2_{L^2}
+\|\na^2\ps\|^2_{L^2})}_{\leq C M(t)}\cdot\|\na^4\ps\|_{L^2}\\
&\leq  C\ep \sqrt{M(t)} \|\nabla^3\ps\|^2 +\frac{\ep}{4}\|\nabla^4\ps\|^2.
\end{align*}

 \fi
 
 Collecting the above estimates and using \eqref{lemma2_}, we get
 the third order version of \eqref{beforeclaims_lemma23_}:
\begin{align}\label{beforeclaims34_lemma23_} &\int \fn{ |\na^3 \vp|^2} +\int |\na^3 \ps|^2 + \int_0^t \int \fn{ |\na^4 \vp|^2} %+ \ep\int_0^t \int{ |\na^4 \ps|^2}
\\
 &\le C( \|  \vp_0\|_{3,w}^2  + \| \na^3 \ps_0\|^2+
 \| \na \ps_0\|_{1,w}^2+ \| \ps_0\|^2) 
 %+C\int_0^t \int N |\na^3 \ps|^2 
 +C\sqrt{M(t)}\int_0^t \int  \fn{|\na^3 \ps|^2}.   \nonumber
 \end{align}  
 
 As before, we claim the following two estimates for the third order derivatives:
 \begin{align} \label{claim3_lemma23_}
&\int_0^t \int N |  \na^3 \ps|^2 
% +\ep\int_0^t \int  |\na^2\Del\ps|^2 
  \le   C  ( \|  \vp_0\|_{2,w}^2  +
 \| \na \ps_0\|_{1,w}^2+ \| \ps_0\|^2+\|\na\Del\ps_0\|^2) 
+ C\sqrt{M(t)} \int_0^t \int \fn{|\na^3 \ps|^2}, \\
& \int \fn {|\na^3  \ps|^2} +   \int _0^ t \int \fn {|\na^3 \ps|^2 } 
%+ \ep\int _0^ t \int \fn {|\na^4 \ps|^2 } 
\le  C    ( \| \na \ps_0 \|_{2,w}^2 + \|  \ps _0\|^2 + \|  \vp_0\|_{2,w}^2)
+ C\int_0^t \int \fn{ |\na^4\vp|^2}. \label{claim4_lemma23_}
\end{align}

{These two claims can be proved similarly so we skip the proofs.}

Now we can finish the proof of  Lemma \ref{lemma23_} fully. Indeed
plugging \eqref{claim4_lemma23_}  into \eqref{beforeclaims34_lemma23_} with $M(t)$ small, we  have
\begin{align*}
 &\| \na^3\vp\|_w^2 + \| \na^3\ps\|^2 + \int_0^t \| \na^4 \vp\|_w^2
%   + \ep \int_0^t\|\na^4\ps\|^2 
  \le  C ( \| \na \ps_0\|_{2,w}^2 + \| \ps_0\|^2 + \| \vp_0\|_{3,w}^2). %+ C{M(t)}   \int_0^t \int \fn{|\na^2 \vp|^2}
\end{align*}
%Thus the first estimate of Lemma \ref{lemma1_} holds for a sufficiently small  $M(t)$.  
So from \eqref{claim4_lemma23_},  we have
\begin{align*}%\label{claims_lemma1_}
\int \fn{ |\na^3 \ps|^2} + \int_0^t \int \fn{ |\na^3 \ps|^2} 
%+  \ep\int_0^t \int \fn{ |\na^4 \ps|^2}
 \le C ( \|  \vp_0\|_{3,w}^2   + \| \na \ps_0\|_{2,w}^2
+ \| \ps_0\|^2).
\end{align*}   This proves Lemma \ref{lemma23_} for case $k=i+j=3$.
 \end{proof}
 
Finally, we obtain Proposition  \ref{uniform_}:
\begin{proof}[Proof of Proposition \ref{uniform_}]

  Adding  Lemma \ref{lemma01_} and Lemma \ref{lemma23_} for $k=2,3$, we have\\
  \begin{equation}\begin{split}\label{lemma3_}
&  {\sup_{s\in[0,t]}\| \vp(s)\|_{3,w}^2     + \| \na \ps(s)\|_{2,w}^2 + \| \ps(s)\|^2} + \int_0^t  \sum_{l = 1}^4 \| \na^{l} \vp\|_w^2
+\int_0^t   \sum_{l = 1}^3\| \na^l \ps\|_w^2
%+\ep\int_0^t     \| \na^{4} \ps\|_w^2
\\
&   \le
 C  { ( \|  \vp_0\|_{3,w}^2  +
 \| \na \ps_0\|_{2,w}^2+ \| \ps_0\|^2)}.
\end{split}\end{equation}  This proves Proposition  \ref{uniform_}. 
\end{proof}
 \iffalse
\begin{remark} Note that the above estimate \eqref{lemma3_} does not contain any control of  $\|\ps\|_w$. This is not a problem when $\ep>0$. Indeed, for $\ep>0$, thanks to the Dirichlet condition (in $y$), we can use the Poincar\'{e} inequality (see \eqref{ineq:poincare3}) to get

$$\| \ps\|^2_w=\int \frac{|\ps|^2}{N}\leq C\la^2\int \frac{|\ps_y|^2}{N}\leq \| \na \ps\|_{w}^2\leq 
C  { ( \|  \vp_0\|_{3,w}^2  +
 \| \na \ps_0\|_{2,w}^2+ \| \ps_0\|^2)}.$$  However, when $\ep=0$, the periodic condition (in $y$) is not enough to get such an estimate.
\end{remark}
\fi
 \ \\
 
\subsection{Proof of Proposition \ref{zerolocal}}\label{pf_local_exist}

%{\it Proof of Proposition \ref{zerolocal}}\\
\indent
Let us remind the equation \eqref{eq0:phipsi}:
%\begin{align}\label{eq00:phipsi} \begin{aligned}
%\varphi_t - s \varphi_z - \Delta \varphi
%& =   N \na \psi + P\na \cdot \varphi + \na\cdot \varphi \na \psi \\
%\psi_t - s \psi_z -\ep \Delta \psi &=  -2\ep  P \cdot\na \psi -\ep  |\na \psi|^2 + \dv\varphi
%\end{aligned}
%\end{align}
\begin{align}\label{eq00:phipsi} \begin{aligned}
\varphi_t - s \varphi_z - \Delta \varphi
& =   N \na \psi + P\na \cdot \varphi + \na\cdot \varphi \na \psi, \\
\psi_t - s \psi_z 
 &=  
  \dv\varphi
\end{aligned}
\end{align}
with data $(\varphi_0, \psi_0) $ satisfying $ \| \varphi_0\|^2_{H^3_{w,p}} + \| \psi_0\|^2_{H^3_p} + \|\na \psi_0\|^2_{H^2_{w,p}}  < M $. 
\indent
First we shall find the local solution $\Phi = (\varphi, \ps)$ of \eqref{eq00:phipsi} 
in 
\[X_T = \{ \Phi \,|\, \vp \in L^{\infty}((0, T): H^3_p) \cap L^2((0, T): \dot{H}_p^4), \psi \in L^{\infty}((0, T): H^3_p)\}\]
%\[X_T= L^{\infty}([0, T): H^3_{p} \times H^3_{p} \times  H^3_{p}) 
%\bigcap L^2([0, T); H^4_p\times H^4_p \times H^4_p)\]
\[ \| \Phi\|_{X_T} = \sup_{ 0<t <T} \| \Phi\|_{H^3_p} + \int_0^T \| \na \vp\|_{H^3_p} ds\]
 for a sufficiently small $T$.
%We shall find the local solution in  $L^{\infty}([0, T]: H^3_{0})$ at first, then show the regularity to persist
%if initial data is such that $\varphi\in H^{3}_{w,0}$ and $\psi \in H^3_{0}$ and $ \nabla \psi \in H^2_{w, 0}$.

Step 1. We define the linear map
$\mathcal F:X_T \to X_T $ mapping  
${\tilde\Phi:=}\color{black}( \tilde \varphi, \tilde \ps)$ to ${{\Phi}:=}\color{black}(\varphi, \ps)$.
%\[Y_T= L^{\infty}((0, T): H^3_{0}\times H^3_0 \times H^3_0)
% \bigcap W^{1, \infty}((0, T); H^1_0 \times H^1_0\times H^1_0).\]
%\begin{align}\label{eq:linear}\begin{aligned}
%\varphi_t - s \varphi_z - \Delta \varphi
%& =   N \na \psi )+ P\na \cdot \varphi + \na\cdot \tilde \varphi \na \tilde \psi \\
%\psi_t - s \psi_z -\ep \Delta \psi &=  -2\ep  P \cdot\na \psi -\ep  | \na \tilde \ps|^2 + \dv\varphi\\
%(\varphi, \psi)(0)& = (\varphi_0, \psi_0).
%\end{aligned}\end{align}
\begin{align}\label{eq:linear}\begin{aligned}
\varphi_t - s \varphi_z - \Delta \varphi
& =   N \na \psi + P\na \cdot \varphi + \na\cdot \tilde \varphi \na \tilde \psi, \\
\psi_t - s \psi_z
  &=  \dv\varphi,\\
(\varphi, \psi)|_{t=0}& = (\varphi_0, \psi_0).
\end{aligned}\end{align}
We approximate the above  linear system by Galerkin method. Since $H^3_p (\Omega)$ is a separable 
Hilbert space, there exists an orthonormal basis
 $  \om_j, j=1, 2, 3 \dots$ of $H^3_p$. 
We define $\Phi^k = ( \varphi^{k1}, \varphi^{k2}, \psi^k)^t $ by
\[ ( \varphi^{k1},\varphi^{k2}, \psi^k )^t= \sum_{j=1}^k  \om_j g^{jk}(t),\] 
where   $g^{jk} = (g^{jk}_1, g^{jk}_2, g^{jk}_3)^t $ solves the following ode system
\begin{equation}\label{eq:ode}
 \ddt \llangle \Phi^k, \om_j\rrangle +  \llangle B\Phi^k, \om_j \rrangle = 
 \llangle f, \om_j \rrangle
\end{equation}
with
%\[B =   \begin{pmatrix}
%-\Del - P \dv -s\partial_z & -N\nabla \\ -\dv & -\ep \Del+2\ep P\dv -s\partial_z.
%\end{pmatrix}  \]
\[B =   \begin{pmatrix}
-\Del -s\partial_z +\calP  \pa_z & 0 & -N\pa_z \\ 0& -\Del -s\partial_z & -N\pa_y \\
-\pa_z &-\pa_y& -s\partial_z
\end{pmatrix}, \quad f = \begin{pmatrix}  \dv \tilde \vp \pa_z \tilde \ps\\
 \dv \tilde \vp \pa_y \tilde \ps\\ 0 \end{pmatrix}.  \]
 Here, $\llangle \Phi, w \rrangle$  is the vector valued function defined by 
 \[\llangle \Phi, w \rrangle =
  ( \langle \Phi_1, w\rangle,   \langle \Phi_2, w\rangle,  \langle \Phi_3, w\rangle)^t,\]
  where  $ \langle \cdot, \cdot \rangle$ is the $H^3$ inner product of scalar valued functions, that is 
  \[ \langle f, g \rangle = \sum_{i+j \le 3}\int \pa_y^i \pa_z^j f \overline{\pa_y^i \pa_z^j g} dydz.\]
  The initial data $g^{jk}(0)$ is determined by $\Phi_0 = \sum_{j=1}^{\infty} c^{j}w_j$ such that
 \[ g^{jk}(0) = c^j \quad  0\le j\le k \mbox{ for any } k\ge 0.\]
The \eqref{eq:ode} 
determines the $ 3k \times 3k $ system of ordinary differential equations with respect to $ g_l^{jk},  1\le j\le k, l= 1, 2, 3$,
%of the form of 
%\[\pa_t
%\begin{pmatrix} g_l^{jk} \end{pmatrix} + A \begin{pmatrix} g_l^{jk} \end{pmatrix} = F\]
%for a constant matrix $A$ and a given source $F$. ,
hence $g_l^{jk}$ exists globally.
By multiplying $(g^{jk})^t$ on the left to \eqref{eq:ode} and summing over $j$, we have also 
\begin{equation}\label{eq:kthinner}
\frac{1}{2}   \ddt \langle \Phi^k , \Phi^k \rangle  +  \langle B\Phi^k, \Phi^k \rangle  =
 \langle f, \Phi^k \rangle .
\end{equation}
(The inner product of vectors are defined componentwise.)\\

Step 2.  We apply the integration by parts to the term $\langle B\Phi^k, \Phi^k \rangle$, and by usual energy estimates we arrive at 
\[ \| \Phi^k\|^2_{H^3_p} + \int_0^ t \| \na \varphi^k \|^2_{H^3_p} ds \le C \| \Phi_0\|^2_{H^3_p} +
C \int_0^t (\|\Phi^k\|_{H^3_p} + \| \na \tilde \vp\|_{H^3_p}^2  + \| \tilde \Phi\|_{H^3_p}^2 )ds\]
 with $C$ depending on $\| (\calP, N)\|_{L^{\infty}}$.
   Finally by Gronwall's inequality, we have the uniform estimate of  $\Phi^k$ in $X_T$ for a small $T$ depending only on $\Phi_0$ and $\|\Phi\|_{X_T}$. Then
$\Phi^k$ converges weakly to $\Phi\in X_T$, which is  a distribution solution of \eqref{eq:linear}. 
Note that \eqref{eq:linear} is linear, so the weak convergence is sufficient for passing to the limit. Moreover
$\Phi$ satisfies 
\[ \ddt \Phi + B\Phi = f \,\, a.e \,\mbox{ on } (0, T), \quad \Phi(0)= \Phi_0.\]
%(\comment{see Temam Navier-Stokes equations p.191}). 
for which we refer to Theorem $1.1$ in Chapter $3$ of \cite{Te}.
 From the equation above, we have also $\Phi = (\vp, \ps) \in 
W^{1, \infty}((0, T): H^1_p\times H^1_p)$. In particular, $\Phi$ is in $C([0, T): H^3_p\times H^3_p)$.
%Note that \eqref{eq:linear} is linear, so the weak convergence is sufficient. 
%Then $\Phi^k$ is compact in $L^2((0, T): B_0)$ for $B_0$ compactly embedded in $H^3_0$. We may take $B = H^1$ and $B_0 = H^1_{0,loc}$. This $\phi$ solves
% By further energy estimate, we hopf to show 
% $\color{red} \Phi  =\color{black}(\varphi, \psi)\in L^{\infty}((0, T): H^3_{0}\times H^3_0)$, $\na \Phi \in   L^2((0, T): H^{3}_0 \times H^{3}_0))$ such that 
%\[ \ddt ( \color{red}  \Phi   \color{black}, \Psi) + ( B\color{red}  \Phi   \color{black}, \Psi) = (f, \Psi)  \mbox{ for all } \Psi\in H^3_0, \]
%and $\color{red} \Phi  \color{black}(0)= (\varphi_0, \psi_0)$ \\
%$\bullet$ Moreover, the distribution solution is unique. It is actually the fuction in
%$C([0, T]: H^3_{0}\times H^3_0) \bigcap L^2([0, T]; H^4_0 \times H^4_0)$.

Step 3.
We can write \eqref{eq00:phipsi} into the following Duhamel's formula.
\[ \Phi = S(t) \Phi_0 + \int_0^t S(t-s)( \mbox{ RHS of }\eqref{eq:linear}) (s) ds,\]
where 
$S(t)\Phi_0 = ( e^{-t\Del}\varphi_0, e^{-\ep t\Del} \ps_0)^t$.
The injectivity of the map $\mathcal F$ is clear from the formula above.
We proceed as usual. Let $\| \Phi_0\|_{H^3_p \times H^3_p} \le M/2$ and $\| \tilde \Phi\|_{X_T} \le M$.
We can show by heat kernel estimates
\[ \| \Phi\|_{X_T} \le M, \quad 
\| \Phi - \Psi\|_{X_T} \le \frac 12 \| \tilde \Phi - \tilde \Psi\|_{X_T} \]
if $T$ is sufficiently small. By the fixed point theorem we obtain the  solution of \eqref{eq00:phipsi} in $X_T$.
Then the similar energy estimates as presented in Section 3 show that the solution $(\varphi, \psi)$ satisfies 
$ \varphi \in H^3_{w,p}, \na \ps \in H^3_{w, p}, \psi \in H^3_p, \na \psi \in H^2_{w,p}$
 and \eqref{double} holds, for which we omit details.
  
%For the $\ep=0$ case, we let 
%\[ \varphi = \sum_{n\in \mathbb Z} \varphi_n (x,t) e^{2\pi i n/\la y}, \quad 
%\ps =  \sum_{n\in \mathbb Z} \ps_n (x,t) e^{2\pi i n/\la y}.\]
%Substituting the above into \eqref{eq:linear}, we obtain the infinite system of ode in terms of the $n$th fourier coefficient. First solve $(\varphi_n, \psi_n)$ and  define $\Phi= (\varphi, \psi)$ from $ \tilde \Phi=(\tilde \varphi, \tilde \psi)$. By energy estimates to \eqref{eq:linear}, show $\Phi \in X_T$. Then writing 
%\eqref{eq:linear} into the Duhamel's formula, we find the local solution by contraction.
\end{subsection}
%Note that we have used $M(t)$ for $ \| \vp\|_{L^{\infty}} \le M(t)$, $ \| \dv\vp\|_{L^{\infty}} \le M(t)$, and  $ \|\na \ps\|_{L^{\infty}} <M(t)$.

\ \\

\section{Proof of Theorem \ref{theoremli}}\label{sectionuniformwith}
In this section we introduce the proposition \ref{uniform_with} below which implies
Theorem \ref{theoremli}. 
\begin{proposition}\label{uniform_with}
There exist constants 
$\ep_0>0,\,
% \delta_0>0,\, 
 C_0$, and $\la_0>0$ such that
for $ 0< \ep \leq \ep_0$,
and 
  $\la\in(0,\la_0]$, we have the following:\\
 If  $(\varphi, \psi)$ be a local solution of \eqref{eq0:linear} on $[0,T]$ for some $T>0$ with $M(T)<\infty$  {for some initial data $(\varphi_0, \psi_0)$ which is $\la$-periodic  in $y$ and which has the zero average in $y$ ($\bar{\varphi_0}=\bar{\psi_0}=0$)}, 
then   we have 
\[   M(T) +  \int_0^{T}  \sum_{l = 1}^4 \| \na^{l} \vp\|_w^2
+\int_0^{T}   \sum_{l = 1}^3\| \na^l \ps\|_w^2 + \ep\int_0^{T}   \| \na^4 \ps\|_w^2 \le C_0M(0).\]
\end{proposition}

\begin{remark}
\begin{enumerate}
\item
 {For this linear case, we do not need any smallness of $M(T)$ since such a condition can pop up only from the nonlinear term estimate which is not present in this linear case \eqref{eq0:linear}.}
\item As in Proposition \ref{uniform_}, $C_0$ does not depend on the size of $T>0$.
 \end{enumerate}
\end{remark}
For the proof we shall proceed the similar line of arguments as for Propotision \ref{uniform_} with a minimal modification.
First we recall that we assumed $\bar \vp_0=0$ and $ \bar \ps_0=0$. 
Since the coefficient functions in the system \eqref{eq0:linear}   depend on $z$ only, the zero average conditions are preserved over time. It enables us to use the Poincar\'{e} inequality   to $(\vp, \ps)$ with respect to $y-$variable:
\begin{equation*}%\label{poincare} 
\| \vp(z,\cdot_y) \|_{L^2_y([0,\la])} \le C\la \| \pa_{y} \vp(z,\cdot_y)\|_{L^2_y([0,\la])}, \quad
\| \ps(z,\cdot_y) \|_{L^2_y([0,\la])} \le C\la \| \pa_{y} \ps(z,\cdot_y)\|_{L^2_y([0,\la])},
\, % \quad\mbox{ for } 
z\in\mathbb{R}.
\end{equation*}
Second we collect some properties of traveling waves $(N, \mathcal P)$.
Note that $\ep>0$ and $N, P= (\calP, 0) $ solves
\begin{align}\label{npp}\begin{aligned} 
- sN'- N''& =  (N\calP)', \\
-s\calP' -  \ep \calP ''& =  -2\ep \calP  \calP'+  N'
\end{aligned}
\end{align} with the boundary condition \eqref{np_copy_bd}. All the properties in Lemma \ref{wave_prop} hold still true when $\ep>0$.
\begin{lemma}\label{wave_propwith}
%For $s>0$, t
%For the f $(N,\calP)$ of \eqref{np_copy} with $s>0$,
For $\ep_0>0$, %Let $\ep < \frac 14$.
there exists a constant %a constant  $\ep_0>0$
%$0<\delta<1$,
%and 
$M>0$ %independent of $\ep$
 such that for any $0<\ep\leq\ep_0$,
%for any $\ep\in[0,\ep_0]$, 
we have:\\
 $$ \frac{w}{M}\leq\frac{1}{N }\leq Mw,$$
% $$ 
% %(1-\delta)\frac{1}{s}\leq
% \Big|\frac{\calP_\ep}{N_\ep}\Big|\leq (1+\delta)\frac{1}{s}$$
% $$ 0<N_\ep\leq(1+\delta)(1+\ep_0)s^2$$
 %$$-(1+\delta)s\leq \calP_\ep<0$$
 %$$ (\frac{1}{N_\ep})'\leq s\Big(2(1-\delta_2)\frac{1}{N_\ep}\Big)$$
% $$ \Big(\frac{1}{N_\ep}\Big)'\geq0$$
  %$$|\calP'_\ep|<(1+\delta)(1+\epsilon_0)s^2$$
$$ |{N^{ (k)}}| \leq M, \quad   |{\calP^{(k)}}|  \leq M,\quad \mbox{ for } 0\leq k\le 2,\quad \mbox{ and } $$
$$ \frac{N'}{N}=-(\calP+s),   \quad 
\Big|(\frac{1}{N})'\Big|+\Big|(\frac{1}{N})''
\Big|\leq \frac{M}{N},\quad 
\Big|(\frac{1}{\sqrt {N}})'\Big|\leq \frac{M}{\sqrt {N}}.$$
($N^{(k)}$ is any $k-$th derivative of $N$)
\end{lemma}
\begin{proof}
The first inequality follows  Lemma \ref{KPPtheorem}.
 For the second line, it is obvious that  $k=0$  case holds. If $|\mathcal P^{(k)}|<M$ $(k=1,2)$ are shown,
 % which is proven in Lemma $3.4$ \cite{LiLiWa}, 
 the other items follow  the equation \eqref{npp}. In the below we present the proof of $|\mathcal P^{(k)}|<M$ $(k=1,2)$.
\\
\indent $\bullet$ $ k=1$\\
  This case $|\mathcal P'|<C$  was proved in \cite{LiLiWa}. We sketch its proof for readers convenience.\\
  \indent
  % The second equation of \eqref{npp} is in fact $-s\calP- \ep \calP' = -\ep \calP^2 + N$.
  Multiplying $\cpq$ to the second equation of \eqref{npp} and integraing over $(-\infty, z)$, we have
  \begin{align*}
  \frac s2 \cpp^2 + \ep \infz{\cpq^2}& \le \infz {( \ep\cp^2-N)'\cpq} \\
  & =\ep \infz{2 \cp \cpp\cpq } - \infz{N'\cpq}\\
  &= \ep \infz {\cp ((\cpp)^2)'} + \infz{ (-sN' - (N\cp)')\cpp} - N'\cpp \\
  & = \ep \cp \cpp^2 - \ep\infz{\cpp^3} - s \infz{N'\cpp} - \infz{(N\cpp^2+N'\cp\cpp)}- N'\cpp\\
  &\le -s \infz {N'\cpp} + \frac s4 \cpp^2 + \frac{N'^2}{s}
  \end{align*}
  by using monotonicity of $N, \cp$ and  $\cp<0$.
  Next multiplying $\cpp$ to the same equation, we can show 
 \[ \infz{ \cpp^2} \le \frac{1}{s^2} \infz {\np^2}.\]
 On the other hand, from the first equation $\np = -sN - N\cp$, we find
 \[ \infz { \np^2} \le - \| \np\|_{L^{\infty}}\infz{\np} 
% \le CN 
 \le C.\]
 Thus we have
 \begin{align}\label{Ncp}
 \infz{\cpp^2} \le C, \quad \infz{\np^2} \le C.
 \end{align}
 Also note that $\np^2 \le C$.
Plugging the bounds in the inequality
\[  \frac s2 \cpp^2 + \ep \infz{\cpq^2} \le  -s \infz {N'\cpp} + \frac s4 \cpp^2 + \frac{N'^2}{s},\]
we conclude $ |\cpp| < C$.\\

  $\bullet$ $k=2$\\
 First we observe  $\cp'' (-\infty)= N''(-\infty)=0$ from the equations for any $\ep>0$.
Multiplying $\cp'''$ to the equation $-s \cpq - \ep \cp''' = -\ep(\cp^2)'' + N''$, 
and integrating over $(-\infty, z)$, we have
\begin{align}\label{cpq} \begin{aligned}
\frac s2 \cp''^2 + \ep \infz{ \cp'''^2} & = \infz{ ( \ep \cp^2 -N)'' \cp'''}\\
& = 2\ep \infz{ \cpp^2 \cp''' }+ \ep \cp \cpq^2 - \ep \infz{ \cpp \cpq^2}
- \infz {N'' \cp''' }\\
&\le 2\ep \infz {\cpp^2 \cp'''}- \infz {N'' \cp'''}\\
& = 2 \infz{ \cpp^2 ( -s\cpq + (\ep \cp^2 - N)'')} + \infz { N''' \cpq} - N''\cpq.
\end{aligned}\end{align}
We use the monotonicity of $N, \cp$ and $\cp<0$ in the third line. 
We find
\begin{align*}
\infz{ \cpp^2 ( -s\cpq + (\ep \cp^2 - N)'')} & = -\frac{s}{3} { \cpp^3} + \infz{\cpp^2 (\ep \cp^2 - N)''}\\
& \le \infz{\cpp^2 (\ep \cp^2 - N)''}
\end{align*}
which can be shown bounded by substituting $N''$, $\ep \cpq$ with lower order terms following the equations and using  $ N^{(k)} <C$ , $|\cp^{(k)}|<C$ for $k=0, 1$ and \eqref{Ncp}. On the other hand, we find
\begin{align*}
\infz { N''' \cpq}& = \infz {( -s N' - (N\cp)')' \cpq} \\
& = \infz{ -s N'' \cpq} - \infz { (N\cp)'' \cpq} ,
\end{align*}
\begin{align*}
-\infz { (N\cp)'' \cpq}& = -\infz{ N \cpq^2} - \frac 12\infz{ N' (\cpp^2)'} -\infz {\cp N'' \cpq}\\
&\le \frac 12 \infz{ ( -sN' - (N\cp)')\cpp^2} - \frac 12 N' \cpp^2 - \infz{\cp N'' \cpq}
\end{align*}
 The first integral and the second term are bounded. It remains to estimate 
 $\infz {N''\cpq} $ and $N''\cpq$ by the Cauchy-Schwartz inequality. 
 From the first equation, it is easy to see $\infz{ N''^2} < C $.  By multiplying $\cpq$ to the equation 
   $-s \cpq - \ep \cp''' = -\ep(\cp^2)'' + N''$ and integrating over $(-\infty, z)$, we have
   \begin{align*}
   \frac s2 \infz {\cpq^2} + \frac{\ep}{2} \cpq^2 
   & \leq  \ep \infz{ (\cp^2)'' \cpq }+ \frac{1}{2s} \infz {(N'')^2}\\
   & = 2\ep \infz { \cpq^2  \cp  } + \frac{2\ep}{3} (\cpp)^3 + \frac{1}{2s} \infz {(N'')^2}\\
   & \leq  \frac{2\ep}{3} (\cpp)^3 + \frac{1}{2s} \infz {(N'')^2}\leq C.
   \end{align*} %Note that $|\cp| <s$, so if $\ep<1/4$,
    So we have
    $\infz {\cpq^2} \le C$. 
    Finally we estimate
    \begin{align*}
    |N''\cpq| \le \frac s 4 \cpq^2 + \frac{N''^2}{s} \le \frac s4 \cpq^2 + \frac{1}{s} ( sN' + (N\cp)')^2\leq \frac s4 \cpq^2 + C.
    \end{align*}
    Plugging the estimates above into \eqref{cpq}, we obtain
    $|\cpq|<C$.
\end{proof}
Proposition \ref{uniform_with} is proved by the lemmas below which are parallel to Lemma \ref{lemma0_}, \ref{lemma1_0}, \ref{lemma1_1}, 
\ref{lemma1_2}, \ref{lemma01_}, \ref{lemma23_}. Here we present in detail the zeroth order estimate of $\ep>0$ case, where the zero average condition, hence the Poincar\'{e} inequality is essentially used. 
 {For the higher order estimates, we sketch their proof. }%All the informations of traveling wave solution proved in 
%Lemma \ref{wave_propwith} can be used in the higher order estimates. 

\begin{lemma}\label{lemma0_e}
If %$M(T)>0$, 
$\la>0$
%, and $\ep\ge 0$ are
 is sufficiently small, we have the following: % for any $t\in[0,T]$,
\begin{equation}\label{eq_lemma0_e}
\| \psi \|^2 + \| \varphi\|_w^2 + \int_0^t \| \na \varphi\|_w^2
+{\ep}\int_0^t \|\na \psi\|^2
\le C ( \| \psi_0\|^2 + \| \varphi_0\|_w^2 ).% C\sqrt{M(t)} \int_0^t  \|\na \psi\|_w^2.
%\int_\mathbb{R} \fn{|\na \psi|^2}.
\end{equation}
\end{lemma}
 
\begin{proof}
Recall the equations \eqref{eq0:linear}.
Multiply  $\fn{\varphi}$ to the $\varphi$ equation and $\psi$ to the $\psi$ equation. 
Integrating by parts, we have
\begin{align*}
&\frac 12 \ddt \left( \int \fn{ |\varphi|^2} +\int  |\psi|^2 \right) 
+ \int\fn{ \sum_{i} |\na \varphi^i|^2} + \ep\int |\na \psi|^2+ 
\frac s 2 \int |\varphi|^2 \left( \fn{1}\right)'\\
&= \frac 12 \int |\varphi|^2 \left(  \fn{1}\right)''+
\int \fn{\calP} \varphi^1 \na \cdot \varphi  
%+  \int \fn{\varphi \cdot \na\psi} \na\cdot \varphi
-%\underbrace{ 2\ep \int P\cdot \na\ps \psi }_{= 
 2\ep \int \calP \ps_z \psi
 %}
%- \ep \int |\na \ps|^2 \ps.
\end{align*}

  Note that   
\eqref{relation} does not hold for $\ep>0$. Instead, we estimate  the quadratic term $\int \fn{\calP} \varphi^1\na \cdot \varphi$ by
\begin{align*}
 \Big|\int \fn{\calP} \varphi^1\na \cdot \varphi\Big|  & \le   \| \vp^1/\sqrt{N}\|  \| \na\vp/\sqrt{N}\|
  %\le (1+\delta)s \la C \| \pa_y\vp^1/\sqrt{N}\|_2 \| \na\vp\|_w
  \le C \la   \| \na\vp\|^2_w,\\
\end{align*} where we used $|\calP|\leq C
%(1+\delta)s
$ (Lemma \ref{wave_propwith}) with the mean-zero condition  for $\vp_0$ 
($\bar{\vp_0}(z)=0$), which is preserved in time, in order to use  the
Poincar\'{e} inequality \begin{equation}\label{ineq:poincare2} 
\| \vp(z,\cdot_y) \|_{L^2_y([0,\la])} \le C\la \| \pa_{y} \vp(z,\cdot_y)\|_{L^2_y([0,\la])} \quad\mbox{ for } z\in\mathbb{R}.
\end{equation} %Recall that $N$ is independent of $y$.
Also we estimate
\begin{align*}
&\frac 12 \int |\vp|^2 \lr{ \fn{1}}''-\frac s 2 \int |\varphi|^2 \left( \fn{1}\right)'
\le -\int \vp\cdot\vp_z  \left( \fn{1}\right)'+\frac s 2 \int |\varphi|^2 \left|\left( \fn{1}\right)'\right|\\
&\le  C\int |\vp||\vp_z |  \left( \fn{1}\right) +\frac {Cs}{2} \int |\varphi|^2  \left( \fn{1}\right) 
\le  \frac{1}{4}\int  |\vp_z |^2\left( \fn{1}\right)+(C+\frac {Cs}{2} ) \int |\varphi|^2  \left( \fn{1}\right)\\
&\le  \frac{1}{4}\int  |\vp_z |^2\left( \fn{1}\right)+C\la^2(C+\frac {Cs}{2} ) \int |\varphi_y|^2  \left( \fn{1}\right)\le 
   (\frac{1}{4}+C\la^2) \int \frac{|\nabla\varphi|^2 }{N}
 \end{align*} where we used the estimate $|(1/N)'|\leq C/N$ in Lemma \ref{wave_propwith} and the Poincar\'{e} inequality \eqref{ineq:poincare2}.

%By assuming smallness of $\la$, %for any $\ep\ge 0$,
% we arrive at 
%\begin{align*}
%&\frac 12 \ddt \left( \int \fn{ |\varphi|^2} +\int  |\psi|^2 \right) +\frac 18 \int\fn{   |\na \varphi|^2} + \ep\int |\na \psi|^2 \\ 
%& \leq -2\ep \int \calP \ps_z \psi - \ep \int |\na \ps|^2 \ps +C M(t) \int \fn{| \na \psi|^2}
%\end{align*}

For the $\ep$-term, we estimate
\begin{align*}
% - \ep \int |\na \ps|^2 \ps&\leq \ep C \sqrt{M(t)} \|\na\psi\|^2, \\
 -2\ep \int \calP \ps_z \psi &\leq   \ep C \|\psi_z\| \|\psi\| 
 \le    \ep   C \la \|\psi_z\| \|\psi_y\|\le    \ep   C \la  \|\na\psi\|^2%.
\end{align*}
%For the first estimate, we used $\|\psi\|_{L^\infty}\le C\|\psi\|_{H^2}\leq C\sqrt{ M(t)}$.
%For the second estimate, 
where we used  the mean-zero condition for $\psi$    to use the
Poincar\'{e} inequality for $\psi$: \begin{equation}\label{ineq:poincare3} 
%\| \psi \|_{L^2_y([0,\la])} \le C\la \| \psi_{y} \|_{L^2_y([0,\la])} .
\| \ps(z,\cdot_y) \|_{L^2_y([0,\la])} \le C\la \| \pa_{y} \ps(z,\cdot_y)\|_{L^2_y([0,\la])} \quad\mbox{ for } z\in\mathbb{R}.
\end{equation}

By assuming   %$M(t)$ and
 $\la>0$ small enough, %so that $C M(t)< \frac 14$,
 we arrive at 
\begin{align*}
&\frac 12 \ddt \left( \int \fn{ |\varphi|^2} +\int  |\psi|^2 \right) 
+\frac 18 \int\fn{   |\na \varphi|^2} + \frac{\ep}{2}\int |\na \psi|^2 %+ \frac 12 \int \lr{ \fn{\calP}}' |\vp^1|^2\\
\leq  
%C \sqrt{M(t)} \int \fn{| \na \psi|^2}
0
\end{align*} which proves Lemma \ref{lemma0_e}.
% for $\ep>0$. Note that we used the smallness of $\ep$ to use Lemma \ref{wave_prop}.
 
\end{proof}

%%%%%%%%%%%%%%%%%%%%%%%%%%%%%%%%%%%%%%555
%first order
%%%%%%%%%%%%%%%%%%%%%%%%%%%%%%%55
%First order estimate:
\begin{lemma}\label{lemma1_0_e} If %$M(t)>0$, 
$\la>0$ %, and $\ep\geq0$ are
is small enough, then
we have 
\begin{equation}\begin{split}\label{ineq:lem1__e}
&\| \na \vp\|_w^2  + \| \na \ps\|^2+ \int_0^t \| \na^2 \vp\|_w^2 
+ \ep\int_0^t \|\na^2 \psi\|^2\\
&\le C ( \| \na \vp_0\|_w^2 + \| \na \ps_0\|^2+  \| \psi_0\|^2 + \| \varphi_0\|_w^2 ) 
% \phantom{XXXXX}
+  C \int_0^t \int {N} |\na \ps|^2.% + C\sqrt{M(t)} \int_0^t \int \fn{\abs{\na \ps}}.
\end{split}\end{equation}\end{lemma}
 
\begin{proof}
We repeat the process in the proof of Lemma \ref{lemma1_0}.
For the $\ep$ term, %\begin{align}
%& -2\ep \int (\calP \ps_z)_z \psi_z- \ep \int (|\na \ps|^2)_z \ps_z  - 2\ep \int \calP \ps_{zy} \psi_y- \ep \int (|\na \ps|^2)_y \ps_y,
%\end{align}
 we estimate
\begin{align*}
&-2\ep \int (\calP \ps_z)_z \psi_z  - 2\ep \int \calP \ps_{zy} \psi_y=
  -2\ep \Big(\int \calP' \ps_z \psi_z
-2\ep \int \calP \ps_{zz} \psi_z 
  - 2\ep \int \calP \ps_{zy} \psi_y \Big)\\
  &\leq C\ep \|P\|_{L^\infty}\|\na\ps\|\|\na^2\ps\|+C\ep \|P'\|_{L^\infty}\|\na\ps\|^2   
     \leq C\ep  \|\na\ps\|^2+\frac{\ep}{4}\|\na^2\ps\|^2   
\end{align*} The last term is absorbed into the chemical diffusion term. 
Then we have 
\begin{align*}
&  \ddt \lr { \int \fn{\abs{\na \vp}} + \int \abs{\na \ps} } + \int   \fn{\abs{ \na^2\vp } }
+  {\ep} \int |\na^2 \psi|^2
\le  
C\|{\sqrt{N}} {\na \ps}\|^2  + C   \| \frac{\na \vp}{\sqrt{N}}\|^2 
 +C\ep \|\na \ps\|^2.
\end{align*}
For the last two terms above we use Lemma \ref{lemma0_e}.
\end{proof}

\begin{lemma}\label{lemma1_1_e} If $\la>0$ is small enough, then
we have 
 \begin{align} \label{claim1_lemma1__e}
& \int_0^t \int N |\na \ps|^2 +\ep\int_0^t \int  |\na^2 \ps|^2 \le C \lr{ \| \na \ps_0 \|^2 + \| \ps _0\|^2 + \| \vp_0\|_w^2}. 
%+ C\sqrt{M(t)} \int_0^t  \int \fn{ |\na \ps |^2}.
\end{align}
\end{lemma}
 
\begin{proof}
As in the proof of Lemma \ref{lemma1_1}, we get
 \begin{align}\label{ineq_claim1_lemma1_e}
\int N |\na \ps|^2  
=& \ddt \int \vp \cdot\na \ps - \frac 12 \ddt \int  |\na\ps|^2 \\& \nonumber + \int | \dv \vp|^2 - \int (\dv \vp)P\cdot\na\ps 
%- \int  \dv \vp | \na \ps|^2
\\ \nonumber
&+\ep\int \na( 
  \Del \ps -2  P \cdot\na \ps %-   |\na \ps|^2
)\cdot\na\ps - \ep\int\vp\cdot(   \Del \na\ps  -2 \na( P \cdot\na \ps) %-  \na(|\na \ps|^2)
 )\\& \nonumber=(I)+(II)+(III).
 \end{align}
For $(I)$ and $(II)$, we refer to the proof of Lemma \ref{lemma1_1}.
\iffalse
 Note that 
 \begin{align*}
 {\int_0^t(I)}=\int_0^t\lr{\ddt \int \vp \cdot\na \ps - \frac 12 \ddt \int  |\na\ps|^2}\leq {C(\|\vp(t)\|^2 +  \|\na\ps_0\|^2+\|\vp_0\|^2)}&
 \end{align*} by 
$ \int  \vp \na\ps \le C\| \vp \| ^2 + \frac 12   \| \na \ps\|^2$.

For the second term $(II)$, we estimate 
 \begin{align*}
  \int | \dv \vp|^2-\int (\dv \vp)P\cdot\na\ps 
 %-2s \int \vp_z \na \ps  
 & \le  C\int \fn{ | \na \vp|^2} + \frac 14 \int  N |\na \ps|^2 
  %\int  \dv \vp | \na \ps|^2 & \le C\sqrt{M(t)}    \int {| \na \ps|^2} \le C\sqrt{M(t)}    \int \fn{| \na \ps|^2}.
 \end{align*}
% by bounding $\| \dv\vp \|_{L^{\infty}}  \le C \| \dv\vp \|_{H^2}\le C \sqrt{M(t)}$. 
% $\| \na \ps \|_{L^{\infty}} \le C M(t)$.

\fi
%%%%%%%%%%%%%%%%%%
%\ep estimate begin
%%%%%%%%%%%55
For the $\ep$-term $(III)$,   we estimate \begin{align*}
 \ep \int \na(    \Del \ps -2   P \cdot\na \ps%-   |\na \ps|^2
 )\cdot\na\ps &=
-\ep\int |\na^2\ps|^2
-\ep \int \na(     2   P \cdot\na \ps
%+  |\na \ps|^2
)\cdot\na\ps\\
&\leq-\ep\|\na^2\ps\|^2 +C\ep  \|\na\ps\|^2+\frac{\ep}{4}\|\na^2\ps\|^2\\
%&\leq-\frac{3\ep}{4}\|\na^2\ps\|^2 +C\ep  \|\na\ps\|^2 
\end{align*} and 
\begin{align*}
 - \ep\int\vp\cdot(   \Del \na\ps  -2 \na( P \cdot\na \ps) 
% -  \na(|\na \ps|^2) 
 ) 
&\leq  C\ep\int|\na\vp||\na^2\ps|  + |\na\vp||\na \ps|
% +|\na\vp||\na \ps|^2
 \\
&\leq \frac\ep 4\|\na^2\ps\|^2 +C\ep  \|\na\vp\|^2+C\ep  \|\na\ps\|^2.
%+C\ep M(t)  \|\na\ps\|^2
%\\
%&\leq \frac\ep 4\|\na^2\ps\|^2 +C  \|\na\vp\|_w^2+C\ep  \|\na\ps\|^2.
\end{align*}

Integrating \eqref{ineq_claim1_lemma1_e} in time, we get
 \begin{align*} 
\int_0^t\int& N |\na \ps|^2  
%=& \int_0^t (I)+\int_0^t(II)+\int_0^t(III)\\
\leq 
C(\|\vp(t)\|^2 +\|\na\ps_0\|^2+\|\vp_0\|^2)\\
& +\int_0^t\Big(C  \|\na\vp\|_w^2+ \frac 14 \int  N |\na \ps|^2  
%+C\sqrt{M(t)}    \int \fn{| \na \ps|^2} 
- \frac\ep 2\|\na^2\ps\|^2 +C\ep  \|\na\ps\|^2\Big).
 \end{align*}

%%%%%%%%%%%%%%%%%%
%\ep estimate end
%%%%%%%%%%%55

By Lemma \ref{lemma0_e}, we have 
% \begin{align*} 
% \int_0^t \int N |\na \ps|^2 
% +\ep\int_0^t \int |\na^2 \ps|^2 
% \le& C \lr{ \| \na \ps_0 \|^2 + \| \ps _0\|^2 + \| \vp_0\|_w^2} 
%+ C\sqrt{M(t)} \int_0^t  \int \fn{ |\na \ps |^2}.
% \end{align*}
%This proves the claim 
\eqref{claim1_lemma1__e}.  
\end{proof}

\begin{lemma}\label{lemma1_2_e} If  $\la>0$  and $\ep>0$ are small enough, then
we have 
 \begin{align} 
& \int \fn {|\na \ps|^2} +   \int _0^ t \int \fn {|\na \ps|^2 }
+ \ep \int_0^t\int    \fn{|\na^2\ps|^2}
\le  C(  \| \na \ps_0 \|_w^2 + \| \ps _0\|^2 + \| \vp_0\|_{w}^2) + C\int_0^t \int \fn{ |\na^2\vp|^2}. \label{claim2_lemma1__e}
\end{align}
\end{lemma}
\begin{remark}
Here we used the smallness condition on $\ep$ first in order to get the weighted $L^2$-estimate for 
$\nabla \psi$ above as in Lemma \ref{lemma1_2}.
\end{remark}
\begin{proof}
First we take $\na$ to the $\ps$-equation then multiply by  $ w \na \ps$ to get
%Multiplying $ w \na \ps$ to the equation $$ \na \ps_t - s \na\ps_z -\ep \Del \na\ps = \na (\dv \vp) -2\ep \na( P \cdot\na \ps) -\ep  \na(|\na \ps|^2), $$ we have
\begin{align*} 
&\frac 12 (w |\na \ps|^2)_t - \frac s2 ( w |\na \ps|^2)_z + \frac s2 w' |\na \ps|^2 \\
&= w \na (\dv \vp) \cdot \na \ps 
+\underbrace{\ep w \na \ps\cdot\lr{\Del \na\ps -2  \na( P \cdot\na \ps) }}_{\ep\mbox{-terms}}.
\end{align*}
 \iffalse
Note that for some $c>0$
\begin{align*}
w' = se^{sz} > c w  & \quad \mbox{for } z > 0, \\
w = 1+e^{sz} \le 2
< \frac{1}{c}N 
 & \quad \mbox{for } z\le0.
\end{align*}
Integrating on each half strip %(notation : $\int_{z>0}f:=\int_0^\infty\int_{[0,\la]} f (z,y,t)dy dz$)
  and in time then summing these two,
  \fi
  
  As in the proof of Lemma \ref{lemma1_2}, we get 
 \begin{align*}
&\int w |\na \ps|^2 + \frac c2 \int_0^t \int w |\na \ps|^2 \le \int w |\na \ps_0|^2 +  \int_0^t \int N|\na\ps|^2 + 
C\int_0^t \int w |\na^2 \vp|^2  +\int_0^t \int {\ep\mbox{-terms}}\\
&\le C ( \| \na \ps_0\|_w^2 + \| \ps_0\|^2 + \| \vp_0\|_w^2)
+ C \int_0^t \int \fn{ |\na^2 \vp|^2}+\int_0^t \int {\ep\mbox{-terms}}, 
\end{align*}
where we used the previous estimate \eqref{claim1_lemma1__e}.   

%%%%%
%\ep begins
%%
For the $\ep ${-terms},  we estimate
\begin{align*}
&\int {\ep \mbox{-terms}}=\ep\int { w \na \ps\cdot\lr{\Del \na\ps -2  \na( P \cdot\na \ps)      }}\\
%&=\ep\int {w \na \ps\cdot \Del \na\ps + w \na \ps\cdot\lr{  -2  \na( P \cdot\na \ps) -   \na(|\na \ps|^2)}}\\
&=\ep\int \Big({ -w|\na^2\ps|^2 -w' \na \ps\cdot  \na\ps_z
- w \na \ps\cdot\lr{  2  \na( P \cdot\na \ps)} 
}\Big)\\
&\leq-\ep\int {  w|\na^2\ps|^2 +C\ep\int \Big( w |\na \ps|| \na^2\ps| + w| \na \ps|( |\na^2\ps|+|\na\ps| )
  } \Big)\\
&\leq-\frac \ep 4\int   w|\na^2\ps|^2 +C\ep\int w |\na \ps|^2 
  %\leq-\frac \ep 4\int   w|\na^2\ps|^2 +C\ep\int w |\na \ps|^2 
%  \\& \leq-\frac \ep 4\int   w|\na^2\ps|^2 +\frac{1}{2}\int w |\na \ps|^2 
\end{align*} where we used the estimate $|w'|\leq C|w|$. % and for the last inequality, we assumed %$\ep$ and
 %$M(t)$ small enough.

%%%%%
%\ep ends
%%%%
Plugging this estimate, we have 
\begin{align*}
&\int w |\na \ps|^2 +  (\frac{c}{2}-C\ep
%+\sqrt{M(t)}
) \int_0^t \int w |\na \ps|^2 +\frac{ \ep}{4} \int_0^t\int   w|\na^2\ps|^2 \\
&\le C ( \| \na \ps_0\|_w^2 + \| \ps_0\|^2 + \| \vp_0\|_w^2) %+ C\sqrt{M(t)} \int_0^t \int \fn{ |\na \ps|^2}
+ C \int_0^t \int \fn{ |\na^2 \vp|^2}.  
\end{align*} Then, by making $\ep$ small enough, it proves the estimate \eqref{claim2_lemma1__e}.

\end{proof}

 % a closed energy estimate up to the first order derivatives:
 By adding all the above lemmas, we get
\begin{lemma}\label{lemma01__e} If  $\la>0$  and $\ep>0$ are small enough, then
we have 
  \begin{equation}\begin{split}\label{eq_lemma01__e}
\| \vp\|_{1,w}^2 + \| \ps\|^2    + \| \na \ps\|_w^2+  \int_0^t  \sum_{l = 1,2} \| \na^{l} \vp\|_w^2
+\int_0^t   \| \na \ps\|_w^2
+\ep\int_0^t   \| \na^{2} \ps\|_w^2
\\
\quad  \le
 C  ( \|  \vp_0\|_{1,w}^2  +
 \| \na \ps_0\|_w^2+ \| \ps_0\|^2). 
\end{split}\end{equation}
\end{lemma}
\iffalse
\begin{proof}
Plugging the estimates \eqref{claim1_lemma1__e}  into \eqref{ineq:lem1__e}, 
we  have
 
\begin{align}\label{ineq:lem1_closed_e}
 &\| \na\vp\|_w^2 + \| \na\ps\|^2 + \int_0^t \| \na^2 \vp\|_w^2  + \ep \int_0^t\|\na^2\ps\|^2
\le  C ( \| \na \ps_0\|^2 + \| \ps_0\|^2 + \| \vp_0\|_{1,w}^2) 
\end{align}
 
In addition, from the estimate \eqref{claim2_lemma1__e} with the above estimate \eqref{ineq:lem1_closed_e}, we get
\begin{align}\label{claims_lemma1__e}
\int \fn{ |\na \ps|^2} + \int_0^t \int \fn{ |\na \ps|^2} + 
 \ep\int_0^t \int \fn{ |\na^2 \ps|^2} \le C ( \|  \vp_0\|_{1,w}^2   + \| \na \ps_0\|_w^2
+ \| \ps_0\|^2).
\end{align}
  Adding Lemma \ref{lemma0_e} and the above \eqref{ineq:lem1_closed_e} and \eqref{claims_lemma1__e}, we have \eqref{eq_lemma01__e}.

\end{proof}
\fi
 Then we can repeat this process up to the highest order: 
\begin{lemma}\label{lemma23_e}If %$M(t)>0$,
 $\la>0$  and $\ep> 0$ are small enough, then 
for $k=2, 3$ we get
\begin{align*}
&\| \na^k \vp\|^2_w + \| \na^k \ps\|^2_w +   \int_0^t  { \|\na^{k+1}\vp\|_w^2}
+ \int_0^t  { \|\na^{k}\vp\|_w^2}
+  \ep\int_0^t  { \|\na^{k+1}\ps\|_w^2}\\
&\le  C (  \| \vp_0\|_{k,w}^2+
 \| \na \ps_0\|_{k-1,w}^2+ \| \ps_0\|^2).
\end{align*}
\end{lemma}
We skip the proof(refer to that of Lemma \ref{lemma23_}). 

 \ \\
Finally, we obtain Proposition  \ref{uniform_with}: % by adding all the above lemmas:
\begin{proof}[Proof of Proposition \ref{uniform_with}]

  Adding all the lemmas in this section,  we have\\
  \begin{equation}\begin{split}\label{lemma3_e}
&  {\sup_{s\in[0,t]}\Big(\| \vp(s)\|_{3,w}^2     + \| \na \ps(s)\|_{2,w}^2 + \| \ps(s)\|^2}\Big) + \int_0^t  \sum_{l = 1}^4 \| \na^{l} \vp\|_w^2
+\int_0^t   \sum_{l = 1}^3\| \na^l \ps\|_w^2
+\ep\int_0^t     \| \na^{4} \ps\|_w^2
\\
&   \le
 C  { ( \|  \vp_0\|_{3,w}^2  +
 \| \na \ps_0\|_{2,w}^2+ \| \ps_0\|^2)}.
\end{split}\end{equation}%  This proves Proposition  \ref{uniform_with}. 
\end{proof}
 
\begin{remark} Note that the above estimate \eqref{lemma3_e} does not contain any control of  $\|\ps\|_w$. This is not a problem when $\ep>0$. Indeed, for $\ep>0$, thanks to the mean-zero assumption  (in $y$), we can use the Poincar\'{e} inequality (see \eqref{ineq:poincare3}) to get

$$\| \ps\|^2_w=\int \frac{|\ps|^2}{N}\leq C\la^2\int \frac{|\ps_y|^2}{N}\leq \| \na \ps\|_{w}^2\leq 
C  { ( \|  \vp_0\|_{3,w}^2  +
 \| \na \ps_0\|_{2,w}^2+ \| \ps_0\|^2)}.$$  However, when $\ep=0$, just having the periodic condition (in $y$) is not enough to get such an estimate.
\end{remark}

  \section{Proof of Theorem \ref{thm_eventual}}%\label{stabilitythin}
  %Stability on the thin cylinder

Taking a $y$-derivative  in \eqref{nq} we have 
\begin{align*}
\pa_t n_y - \Del n_y & = \dv (nq)_y \\
\pa_t q_y - \ep \Del q_y & = -2\ep ( (q \cdot \na) q )_y + \na n_y.
\end{align*}
Integration by parts gives 
\begin{align}
\frac 12 \ddt \| n_y\|^2 +  \|\na n_y\|^2  & \le \| q\|_{L^{\infty}} ( \del^{-1}\| n_y\|^2 + \del \| \na n_y\|^2)+ \|n\|_{L^{\infty}} ( \del^{-1} \| q_y\|^2 + \del \|\na n_y\|^2) \label{ny}\\
\frac 12 \ddt \|q_y\|^2 + \ep \|\na q_y\|^2 &\le -2\ep\int ( (q\cdot \na )q)_y \cdot q_y + \int \na n_y \cdot q_y
\label{qy}
\end{align}
where $\del>0$ is small parameter to be chosen.
The right hand side of the second inequality is bounded by 
\[\ep \| \na q\|_{L^{\infty}}\| q_y\|^2 
+ \ep \|q\|_{L^{\infty}}( \del \| \na q_y\|^2 +\del^{-1} \| q_y\|^2) + \del\|\na n_y\|^2 + \del^{-1}\|q_y\|^2.\]
We add \eqref{ny} and \eqref{qy}.
 For sufficiently small $\del>0$, we make $\| \na n_y\|^2$ and $\epsilon\| \na q_y\|^2$  terms  of the right hand side
smaller than $ \frac 12\|\na n_y\|^2 + \frac {\ep}{2} \|\na q_y\|^2 $.
By the Poincar\'{e} inequality (in $y$) we have 
\[ \frac 12\|\na n_y\|^2 \ge \frac{\| n_y\|^2}{C\la^2}, \quad \frac 12 \|\na q_y\|^2 \ge \frac{\| q_y\|^2}{C\la^2}.\]
Then we have 
 \[ \frac 12\ddt ( \|n_y\|^2 +\| q_y\|^2) + \frac{ \|n_y\|^2 +\ep\| q_y\|^2}{C\la^2}  \]
\[\le  
C_1 \del^{-1} (\|n_y\|^2 +\| q_y\|^2) + \ep C_1 (1 +\del^{-1})\|q_y\|^2 + \del^{-1} \|q_y\|^2.\]
Thanks to $\epsilon>0$, we decrease $\la>0$ to get
\[  \ddt ( \|n_y\|^2 +\| q_y\|^2) + \frac{ \|n_y\|^2 +\ep\| q_y\|^2}{C\la^2}  \le 0.\] Thus, there exists a constant $c=c(\ep,\la)>0$ such that
\[  \ddt ( \|n_y\|^2 +\| q_y\|^2) \leq -c ({ \|n_y\|^2 +\| q_y\|^2}).\]

\section*{acknowledgement}

MC's work was supported by NRF-2015R1C1A2A01054919.
KC's work was supported by the 2015 Research Fund(1.150064.01) of UNIST(Ulsan National Institute of Science and Technology) and by NRF-2015R1D1A1A01058614.
KK's work was supported by NRF-2014R1A2A1A11051161 and NRF-20151009350.
JL's work was supported by NRF-2009-0083521.
%\color{red} have to write down 


\begin{thebibliography}{00}
\bibitem{AC} R. A. Anderson and M.A,J. Chaplain, {\it  Modelling the growth and form of capillary networks in On Growth and Form: Spatio-Tempotal Pattern formation in Biology}, Wiley, 225--249,  1999.
\bibitem{BT} P.K. Brazhnik, J. J. Tyson, {\it On traveling wave solutions of fisher's equation in two spatial dimensions}, SIAM J. Appl. Math., Vol. 60, no 2, 371--391, 2000.
\bibitem{Const} P. Constantin, A. Kiselev, and L. Ryzhik, {\it Fronts in reactive convection: bounds, stability and instability}, Comm. Pure and Appl. Math., 56, 1781--1803, 2003.
\bibitem{CPZ1} L. Corrias, B. Perthame, H. Zaag, {\it A chemotaxis
model motivated by angiogenesis}, C. R. Math. Acad.Sci. Paris,
336(2); 141--146, 2003.
\bibitem{CPZ2} L. Corrias, B. Perthame, H. Zaag, {\it Global
solutions of some chemotaxis and angiogenesis systems in high space
dimensions}, Milan J. Math., 72(1); 1--28, 2004.
\bibitem{Ce} Bruce Alberts et al., \textit{Molecular biology of the cell}, 4th edition, Garland Science, New York, 2002.
\bibitem{FoFr} M. A. Fontelos,  A. Friedman and B. Hu, \textit{Mathematical analysis of a model for the initiation of angiogenesis,} SIAM. J. Math. Anal., Vol. 33, 1330--1355, 2002.
\bibitem{FrTe}  A. Friedman and J.I.Tello, \textit{Stability of solutions of chemotaxis equations in reinforced random walks}, J. Math. Anal. Appl., Vol. 272, 136--163, 2002.
\bibitem{JinLiWa} H-Y. Jin, J. Li and Z.A. Wang, 
\textit{Asymptotic stability of traveling waves of a chemotaxis model with singular sensitivity}, J. Differential equations, Vol. 225, 193--210, 2013.
\bibitem{LeNe} J.A.Leach and D.J.Needham. \textit{Matched asymptotic expansions in reaction-diffusion theory}, 
Springer monograghs in mathematics, 2004.
\bibitem{Le} H.A.Levine, B. D. Sleeman and M. Nilson-Hamilton. \textit{A mathemacial model for the roles of pericytes 
and macrophages in the onset of angiogenesis: I. the role of 
protease inhibitors in preventing angiogenesis.} Math. Biosci., 168, 77--115, 2000.
\bibitem{LiLi} D. Li and T. Li,  {\it On a hyperbolic-parabolic system modeling chemotaxis}, Math. Models Methods  Appl. Sci. 21, No. 8,1631--1650, 2011.
%\bibitem{LiWang1} T. Li and Z. A. Wang, Nonlinear stability of traveling waves to a hyperbolic-
%parabolic system modeling chemotaxis, SIAM J. Appl. Math. 70 (2009) 1522-1541.
%\bibitem{LiWa} T. Li and Z. A. Wang, Nonlinear stability of large amplitude viscous shock waves of a generalized hyperbolic-parabolic system arising in chemotaxis, Math. Models Methods Appl. Sci. 20 (2010) 1967-1998.
\bibitem{LiLiWa}J. Li, T. Li and Z.A. Wang, \textit{Stability of traveling waves of the Keller-Segel system with logarithmic sensitivity}, Math. Models and Methods in Appl. Sci., Vol 24, no 14, 2819--2849, 2014. 
\bibitem{LiWa}T. Li and Z. A. Wang, {\it Asymptotic nonlinear stability of traveling waves to conservation laws arising from chemotaxis},  J. Differential equations, Vol. 250, 1310--1333, 2011.
\bibitem{Ho1} D. Horstmann. \textit{From 1970 until present: the Keller-Segel model in chemotaxis and its consequences I}, Jahresber. Deutsch. Math.-Verein., 105,  103–-165, 2003.
\bibitem{HoSt} D. Horstmann and A. Stevens, \textit{A constructive approch to traveling waves in chemotaxis}, J. Nonlinear Sci. Vol. 14,  1--25, 2004.
\bibitem{KSb} E. F. Keller and L. A. Segel, \textit{Traveling bands of chemotactic bacteria: A Theoretical Analysis}, J. theor. Biol. 30, 235-–248, 1971.
\bibitem{Mu} J. D. Murray Mathematical Biology I: An Introduction. Springer, Berlin, 3rd edition, 2002.
\bibitem{NaI} T. Nagai and T. Ikeda, \textit{Traveling waves in a chemotactic model}, J. Math. Biol. 30, 169--184, 1981.
\bibitem{Pe} B. Perthame. PDE models for chemotactic movement: parabolic, hyperbolic and kinetic. 
\textit{Applications of Mathematics.}, 49, 539--564, 2004.
\bibitem {Sc} H. Schwetlick, Travelling waves for chemotaxis systems. Proc. Appl. Math. Mech., 2: 476--478, 2003.
\bibitem{Sh} J. A. Shettatt, Travelling Wave solutions of a mathematical model for tumor encapulation, Siam J. Appl. Math. Vol. 60, no. 2,  392--407, 1999.
\bibitem{Sol} V. A. Solonnikov, \textit{On the solvability of boundary and initial boundary value problems fir the Navier-Stokes system in domains with noncompact boundaries}, Pacific J. Math., Vol. 93, no. 2, 443--458, 1981.
\bibitem{Te} R. Temam, \textit{Navier-Stokes equations}, AMS Chelsea publishing, 2001.
\bibitem{Wa} Z. A. Wang, \textit{Wavefront of an angiogenesis model}, DCDS-B, Vol. 17, no. 8, 2849--2860, 2012.
\bibitem{WaHi} Z.A.Wang and T. Hillen, \textit{Shock formation in a chemotaxis model}, Math. Methods, Appl. Sci., Vol. 31, 45--70, 2008.
\end{thebibliography}
\end{document}